\documentclass[12pt]{amsart}

\usepackage
[a4paper,left=1in,right=1in,top=1in,bottom=1in,footskip=.3in]
{geometry}

\usepackage{amssymb}
\usepackage{amsthm}
\usepackage{amsmath}
\usepackage{tikz}
\usepackage{mathrsfs}
\usepackage{mathtools}
\usepackage{extarrows}
\usepackage{MnSymbol}
\usepackage{thmtools}
\usepackage{pgfplots}
\usepackage{hieroglf}
\usepackage{phaistos}
\usepackage{textcomp}
\usepackage{wasysym}
\usepackage {hyperref}

\usetikzlibrary{matrix,arrows,decorations.pathmorphing}

\theoremstyle{definition}
\newtheorem{thm}{Theorem}[section]
\newtheorem{definition}[thm]{Definition}
\newtheorem{example}[thm]{Example}

\newtheorem{cor}[thm]{Corollary}
\newtheorem{lem}[thm]{Lemma}
\newtheorem{rem}[thm]{Remark}

\newtheorem{prop}[thm]{Proposition}
\newtheorem{question}[thm]{Question}

\allowdisplaybreaks

\title[Rokhlin dimension for correspondences]{Rokhlin dimension for $C^*$-correspondences}
\author{Nathanial P. Brown, Aaron Tikuisis, and Aleksey M. Zelenberg}
\date{} 

\begin{document}

\begin{abstract} 
We extend the notion of Rokhlin dimension from topological dynamical systems to $C^*$-correspondences. We show that in the presence of finite Rokhlin dimension and a mild quasidiagonal-like condition (which, for example, is automatic for finitely generated projective correspondences), finite nuclear dimension passes from the scalar algebra to the associated Toeplitz--Pimsner and (hence) Cuntz--Pimsner algebras. As a consequence we provide new examples of classifiable $C^*$-algebras: if $A$ is simple, unital, has finite nuclear dimension and satisfies the UCT, then for every finitely generated projective $\mathcal{H}$ with finite Rokhlin dimension, the associated Cuntz--Pimsner algebra $\mathcal{O} (\mathcal{H})$ is classifiable in the sense of Elliott's Program. 
\end{abstract} 

\thanks{N.B.\ and A.Z.\ were partially supported by NSF grant DMS-1201385. A.T.\ was partially supported by an NSERC Postdoctoral Fellowship and EPSRC grant EP/N00874X/1.}

\address{N.B.\ and A.Z.\ can be reached at Department of Mathematics, Penn State University, State College, PA 16802.} 
\address{A.T.\ can be reached at Department of Mathematical Sciences, Fraser Noble Building, Aberdeen, AB24 3UE. }

\maketitle

\section{Introduction}
The topological notion of covering dimension was extended to the noncommutative context by Winter and Zacharias in \cite{Winter-Zacharias1}. Their  \emph{nuclear dimension} has contributed to the most important and broadly applicable advances in the theory of nuclear $C^*$-algebras in the last 40 years. For example, the Toms--Winter conjecture asserts that finite nuclear dimension is often equivalent to structural properties analogous to those exploited by Connes in his proof of uniqueness of the injective II$_1$-factor (\cite{Connes}). In a remarkable breakthrough, this audacious conjecture was confirmed in the unique-trace case in \cite{Sato-White-Winter}, and has now been confirmed for much broader classes (see \cite{Brown-Bosa-Sato-Tikuisis-White-Winter}). As another stunning example, in 2015 finite nuclear dimension led to the completion of Elliott's Classification Program for the cases of most interest (cf.\ \cite{Elliott-Gong-Lin-Niu}, \cite{Tikuisis-White-Winter}). In short, nuclear dimension has revolutionized the field. 

It is thus important to know which examples have finite nuclear dimension. In the influential paper \cite{Matui-Sato} it was shown that all Kirchberg algebras have nuclear dimension at most three. (In fact, they have dimension one; see \cite{Brown-Bosa-Sato-Tikuisis-White-Winter} and \cite{Ruiz-Sims-Sorensen}.) In \cite{Szabo} Szabo proved that if $\mathbb{Z}^n$ acts freely on a compact metric space $X$ of finite covering dimension, then the associated crossed product $C(X) \rtimes \mathbb{Z}^n$ has finite nuclear dimension. His work was inspired by and depended upon \cite{Hirshberg-Winter-Zacharias}, where the classical measure-theoretic Rokhlin property was exported to the realm of topology. This so-called \emph{Rokhlin dimension} makes sense in the noncommutative context too, so we can study it for $C^*$-dynamical systems. There is mounting evidence that Rokhlin dimension is for $C^*$-dynamical systems what nuclear dimension is for $C^*$-algebras: ubiquitous and fundamental.

For example, it is intimately connected with nuclear dimension via the crossed product construction. One of the main results of \cite{Hirshberg-Winter-Zacharias} was that $A\rtimes_{\alpha} \mathbb{Z}$ has finite nuclear dimension whenever $A$ does and the automorphism $\alpha$ has finite Rokhlin dimension. Since $A\rtimes_{\alpha} \mathbb{Z}$ can be realized as a Cuntz--Pimsner algebra, it is natural to seek an extension of Rokhlin dimension to the context of $C^*$-correspondences, then ask whether the associated Toeplitz--Pimnser and/or Cuntz--Pimsner algebras have finite nuclear dimension? 

That is the subject of this paper. Indeed, we generalize Rokhlin dimension to $C^*$-correspondences in Definition \ref{Rokhlin-Dimension-Def}, then prove the following.

\begin{thm}\label{Main-Theorem}
Suppose that $\mathcal{H}$ is a countably generated $C^*$-cor\-res\-pon\-dence over a separable unital $C^*$-algebra $A$, satisfying a technical quasidiagonal-like condition (see Theorem \ref{main_thm}). Then 
$$
(\dim_{\mathrm{nuc}}(\mathcal{T}(\mathcal{H})) + 1) \le 2(\dim_{\mathrm{nuc}}(A)+1)(\dim_{\mathrm{Rok}}(\mathcal{H})+1).
$$
\end{thm}

The quasidiagonal-like condition is probably unnecessary, and would be if a certain algebra could be shown directly to have finite nuclear dimension (cf.\ Lemma \ref{quasi_hypothesis}). In any case, it is satisfied in many examples, including all finitely generated projective correspondences (see Example \ref{QD-examples}). 

As an application, we provide new examples of classifiable $C^*$-algebras in the sense of Elliott's Program. Thanks to \cite{Elliott-Gong-Lin-Niu} and \cite{Tikuisis-White-Winter}, this amounts to verifying simplicity, finite nuclear dimension and the Universal Coefficient Theorem (UCT) of  \cite{Rosenberg-Schochet}. 

\begin{cor} (Corollary \ref{application2}) \label{application}
Assume $A$ is simple, unital, satisfies the UCT and has finite nuclear dimension. For every finitely generated projective $\mathcal{H}$ with finite Rokhlin dimension, the associated Cuntz--Pimsner algebra $\mathcal{O}(\mathcal{H})$ is also simple, unital, satisfies the UCT and has finite nuclear dimension. 
\end{cor} 

Note that we've substantially generalized the $C^*$-dynamical system case from \cite[Theorem 4.1]{Hirshberg-Winter-Zacharias} because $A\rtimes_{\alpha} \mathbb{Z}$ is the Cuntz--Pimsner algebra over a singly generated projective correspondence. Also, this corollary holds whenever Theorem \ref{Main-Theorem} does, so it's likely true for arbitrary correspondences of finite Rokhlin dimension (and definitely true when the quasidiagonal-like hypothesis is satisfied).

The proof of Theorem \ref{Main-Theorem} follows \cite{Hirshberg-Winter-Zacharias} very closely, at least in spirit. The main technical innovation is finding suitable replacements for the outgoing and incoming maps used in the proof of Theorem 4.1 in \cite{Hirshberg-Winter-Zacharias}. But we also have to worry about the nuclear dimension of the range of the outgoing maps, an easy task in the crossed-product case.  The majority of the paper is devoted to laying out these issues, and resolving them. Once that is done, the proof of finite nuclear dimension is very similar to \cite{Hirshberg-Winter-Zacharias}.

Here is an outline of what follows. In Section \ref{Preliminaries}, we establish notation and give relevant background information on nuclearity, nuclear dimension, Hilbert $C^*$-modules, $C^*$-correspondences, and Cuntz--Pimsner algebras. Section \ref{Rokhlin dimension for correspondences} is about the definition of Rokhlin dimension for $C^*$-correspondences. The heavy lifting is contained in Section \ref{Heavy lifting}, culminating with the proof of our main result (Theorem \ref{main_thm}). In Section \ref{Amalgamated}, we observe a cute application: in certain circumstances, reduced amalgamated free products have finite nuclear dimension. Finally, in Section \ref{classifiability}, we show how to use work of Schweizer (\cite{Schweizer}) to deduce Corollary \ref{application} from Theorem \ref{Main-Theorem}. 
\section{Preliminaries}\label{Preliminaries}
\subsection{Notation and terminology}

Throughout all that follows, we will use the following conventions, assumptions, and notation. 
When we refer to an \emph{ideal} of a $C^*$-algebra, we mean a closed, two-sided ideal. If $A$ is a $C^*$-algebra, let $A_+$, $\text{Ball}_1(A)$, $\mathcal{M}(A)$, and $\mathcal{Z}(A)$ denote the positive cone, unit ball, multiplier algebra, and center, respectively, of $A$. If $x,y\in A$ are self-adjoint, let $x\approx_{\epsilon}y$, $x\perp y$, and $x\perp_{\epsilon}y$  mean $\lVert x - y\rVert < \epsilon$, $xy = 0$, and $xy \approx_{\epsilon} 0$, respectively. If $C,D\subseteq A$, let $C'$ denote the commutant of $C$ in $A$ and let $C\subset_{\scriptscriptstyle{\epsilon}}D$ mean that for every contraction $c\in C$, there is a $d\in D$ satisfying $c\approx_{\epsilon} d$. 
If $\varphi:A \rightarrow B$ is a linear map to a $C^*$-algebra $B$, we will say that $\varphi$ is $\epsilon$-contractive if $\lVert \varphi\rVert < 1 + \epsilon$. We will use the abbreviation c.p.(c.) to mean completely positive (and contractive).
We write $\mathbb{K}$ for the compact operators on $\ell^2(\mathbb{N})$.
For a $C^*$-algebra $A$, we define
\[ A_\infty = \ell^2(\mathbb N,A) / c_0(\mathbb N,A), \]
and view $A$ as a subalgebra of $A_\infty$ in the canonical way (consisting of elements represented by constant sequences).
\subsection{Order zero maps and nuclear dimension}
Throughout this section let $A$ and $K$ be $C^*$-algebras. 

\begin{definition}\label{Order Zero Maps}
A c.p.\ map $\psi:K\rightarrow A$ is \emph{order zero} if it preserves orthogonality: for every pair of positive elements $x_1,x_2\in K$,
\begin{equation*}
x_1\perp x_2 \Rightarrow \psi(x_1) \perp \psi(x_1). 
\end{equation*}
\end{definition}

%
%
%

\begin{definition}
$A$ has \emph{nuclear dimension} at most $n$, written $\dim_{\mathrm{nuc}}(A) \le n$, if for every finite subset $F\subset A$ and $\epsilon > 0$, there is a finite-dimensional $C^*$-algebra $K = K^{(0)}\oplus\cdots\oplus K^{(n)}$, a c.p.c.\ map $\varphi: A\rightarrow K$, and a c.p.\ map $\psi: K \rightarrow A$ satisfying
\begin{enumerate}
\item $\lVert \psi\circ\varphi(a) - a\rVert < \epsilon$ for every $a\in F$, and
\item for each $i = 0,\ldots,n$, the restriction of $\psi$ to $K^{(i)}$ is contractive and order zero. 
\end{enumerate}
\end{definition}

%
%
The following is well-known and underpins many nuclear dimension computations in the literature.

\begin{lem}
\label{nucdimtool}
Fix $m,n \in \mathbb N$.
Let $A$ be a $C^*$-algebra.
Then $\dim_{\mathrm{nuc}}(A)\leq m(n+1)-1$ if, for every finite set $F \subset A$ and $\epsilon>0$, there exists a $C^*$-algebra $B$ of nuclear dimension at most $n$ and c.p.\ maps
\[ A \stackrel{\scriptstyle{\psi}}\longrightarrow B \stackrel{\scriptstyle{\phi}}\longrightarrow A_\infty \]
such that $\psi$ is c.p.c., $\phi$ is a sum of $m$ c.p.c.\ order zero maps, and $\|\psi\phi(a)-a\|<\epsilon$ for $a \in F$.
\end{lem}

\begin{proof}
By \cite[Proposition 2.5]{Tikuisis-Winter}, we only need to show that the inclusion of $A$ into $A_\infty$ has nuclear dimension at most $(m+1)n-1$
To this end, let $F \subset A$ be finite and let $\epsilon>0$.
Find $B,\psi$, and $\phi$ as in the hypotheses, for the finite set $F$ and with $\epsilon/2$ in place of $\epsilon$.
Since $B$ has nuclear dimension at most $m$, we can factor the identity map on $B$, up to $\frac\epsilon{2n}$ on $\psi(F)$, as
\[ B \stackrel{\scriptstyle{\beta}}\longrightarrow K \stackrel{\scriptstyle{\alpha}}\longrightarrow B \]
such that $\beta$ is c.p.c.\ and $\alpha$ is a sum of $(m+1)$ c.p.c.\ order zero maps.
Then the inclusion $A \to A_\infty$ factors, up to $\epsilon$ on $F$, as
\[ A \stackrel{\scriptstyle{\beta\psi}}\longrightarrow K \stackrel{\scriptstyle{\phi\alpha}}\longrightarrow A_\infty, \]
where $\beta\psi$ is c.p.c., and $\phi\alpha$ decomposes as a sum of $n(m+1)$ c.p.c.\ order zero maps.
\end{proof}

If $J\unlhd A$ is an ideal, then by \cite[Proposition 2.9]{Winter-Zacharias2} we have 
\begin{equation*}
\dim_{\mathrm{nuc}}(A) \le \dim_{\mathrm{nuc}}(J) + \dim_{\mathrm{nuc}}(A/J) + 1. 
\end{equation*}
However a slight modification of the proof of \cite[Proposition 5.1]{Kirchberg-Winter} implies the following statement. 

\begin{prop}\label{propextension} 
If $J \unlhd A$ has a quasicentral approximate unit consisting of projections, then $\dim_{\mathrm{nuc}}(A) = \max\{\dim_{\mathrm{nuc}}(J),\dim_{\mathrm{nuc}}(A/J)\}$. 
\end{prop}
\subsection{Correspondences}
Throughout this section let $A$ be a unital $C^*$-algebra. We give a brief overview of $C^*$-modules (all of which are assumed to be over $A$), $C^*$-correspondences, and Cuntz--Pimsner algebras. For comprehensive treatments, see \cite{Brown-Ozawa,Lance,Pimsner}. 

Let $\mathcal{H}$ be a Hilbert $C^*$-module. We say $\mathcal{H}$ is \emph{full} if the set $\langle \mathcal{H},\mathcal{H}\rangle : = \{\langle x,y\rangle \ : \ x,y\in \mathcal{H}\}$ is dense in $A$.We say $\mathcal{H}$ is \emph{free} if it has an orthonormal set of generators; it is \emph{finitely generated projective} if it is an orthogonal direct summand in a finitely generated free module. 

 If $\mathcal{K}$ is another $C^*$-module, we denote by $\mathbb{B}(\mathcal{H,K})$ and $\mathbb{K}(\mathcal{H,K})$ the adjointable and compact operators from $\mathcal{H}$ to $\mathcal{K}$, respectively. 
$\mathbb{K}(\mathcal{H,K})$ is the closed span of operators $e_{x,y}$ (over $x \in \mathcal K, y \in \mathcal H$), where
\begin{equation}
\label{RankOneFormula}
e_{x,y}(z) = x\ldotp \langle y,z\rangle_{\mathcal H}, \quad z \in \mathcal H.
\end{equation}
In the case $\mathcal{H = K}$, we write $\mathbb{B}(\mathcal{H})$ and $\mathbb{K}(\mathcal{H})$. Besides the operator norm topology, there is another natural topology on $\mathbb{B}(\mathcal{H,K})$: a sequence $T_n$ converges \emph{strictly} to $T$ if $T_n(x) \rightarrow T(x)$ and $T_n^*(y) \rightarrow T^*(y)$ for every $x\in \mathcal{H}$ and $y\in \mathcal{K}$. If $\{\mathcal{H}_i\}_{i\in I}$ is a collection of $C^*$-modules, their direct sum $\bigoplus \mathcal{H}_i$, defined as
\[ \{(x_i)_{i \in I} \in \prod_i \mathcal H_i \mid \sum \langle x_i, x_i\rangle \text{ converges in norm}\}, \]
is also a $C^*$-module. If $\mathcal{H}$ is a free $C^*$-module whose orthonormal generators are indexed by a set $I$, then $\mathcal{H}\cong \bigoplus_IA$.

We will need the following important theorem of Kasparov. 

\begin{thm}[Kasparov's stabilization theorem]\label{Kasparov's Stabilization Theorem}
If $\mathcal{H}$ is a countably generated, then there is a countably generated free $C^*$-module $\mathcal{H}'$ satisfying
\begin{equation*}
\mathcal{H}\oplus \mathcal{H}'\cong \mathcal{H}'.
\end{equation*}
\end{thm}

Kasparov's stabilization theorem implies $\mathbb{K}(\mathcal{H})$ is isomorphic to a hereditary subalgebra of $A\otimes \mathbb{K}$. By \cite[Proposition 2.5]{Winter-Zacharias2}, this implies 
\begin{equation}\label{Compact-Dimension}
\dim_{\mathrm{nuc}}(\mathbb{K}(\mathcal{H})) \le \dim_{\mathrm{nuc}}(A). 
\end{equation}

 We say $\mathcal{H}$ is a \emph{$C^*$-correspondence} (or simply a correspondence) if there is a unital injective *-homomorphism $\omega: A\rightarrow \mathbb{B}(\mathcal{H})$. The map $\omega$ is called the left action of $A$ on $\mathcal{H}$ and we will write $\omega_a(x)$ as $a\ldotp x$. The simplest example is the \emph{identity correspondence}, which is the $C^*$-module $A$ (over $A$) with the left action given by left multiplication.  If $\{\mathcal{H}_i\}_{i\in I}$ is a collection of correspondences, their direct sum $\bigoplus \mathcal{H}_i$ is also a correspondence with the left action given by $\omega_{\bigoplus \mathcal{H}_i} = \bigoplus\omega_{\mathcal{H}_i}$. 

We say two correspondences $\mathcal{H}$ and $\mathcal{K}$ are \emph{unitarily equivalent} (and write $\mathcal{H}\approx \mathcal{K}$) if there exists an adjointable map $U:\mathcal{H} \rightarrow \mathcal{K}$ such that 
\begin{enumerate}
\item $U(a\ldotp x) = a\ldotp U(x)$ for every $x\in \mathcal{H}$ and $a\in A$, and 
\item $U^*U = \text{id}_\mathcal{H}$, and $UU^* = \text{id}_{\mathcal{K}}$. 
\end{enumerate}

The algebraic tensor product $\mathcal{H}\odot\mathcal{K}$ of two correspondences $\mathcal{H}$ and $\mathcal{K}$ is naturally a right $A$-module with an $A$-valued semi-inner product given by 
\begin{equation}
\label{TensorIP}
\langle x_1\otimes y_1,x_2\otimes y_2\rangle = \langle y_1, \langle x_1,x_2\rangle\ldotp y_2\rangle. 
\end{equation}
Denote by $\mathcal{H}\otimes \mathcal{K}$ the $C^*$-module obtained from $\mathcal{H}\odot \mathcal{K}$ by separation and completion.
In $\mathcal H \otimes \mathcal K$, the following identity holds
\begin{equation}
\label{TensorBalance}
x \otimes (a\ldotp y) = (x\ldotp a) \otimes y, \quad x \in \mathcal H, y \in \mathcal K, a \in A.
\end{equation}
There is an injective *-homomorphism $\mathbb{B}(\mathcal{H}) \rightarrow \mathbb{B}(\mathcal{H}\otimes \mathcal{K})$ given by $T\mapsto T\otimes 1$, where $(T\otimes 1)(x\otimes y) = Tx\otimes y$ (and a *-homomorphism $A'\cap \mathbb{B}(\mathcal{K}) \rightarrow \mathbb{B}(\mathcal{H}\otimes \mathcal{K})$ given by $T\mapsto 1\otimes T$). In particular $\mathcal{H}\otimes \mathcal{K}$ is a correspondence, called the \emph{interior tensor product} of $\mathcal{H}$ and $\mathcal{K}$. Elements of the form $h\otimes k \in \mathcal{H}\otimes \mathcal{K}$ are called \emph{elementary tensors}.

If $\mathcal H$ is a correspondence over a $C^*$-algebra $A$, define
\[ \mathcal H_\infty = \ell^\infty(\mathbb N, \mathcal H)/c_0(\mathbb N,\mathcal H), \]
and one can easily check that this is a correspondence over the sequence algebra $A_\infty$ in the obvious way.
It contains $\mathcal H$ as the subset consisting of elements with constant sequence representatives.
In particular, the product (in either order) of an element of $A_\infty$ and an element of $\mathcal H$ makes sense as an element of this $\mathcal H_\infty$.

For a correspondence $\mathcal H$ over a $C^*$-algebra $A$, a \emph{representation} of $\mathcal{H}$ on a $C^*$-algebra $B$ is a pair $(\pi,\tau)$ consisting of a *-homomorphism $\pi: A\rightarrow B$ and a linear map $\tau: \mathcal{H} \rightarrow B$ satisfying $\tau(a\ldotp x\ldotp b) = \pi(a)\tau(x)\pi(b)$ and $\tau(x)^*\tau(y) = \pi(\langle x,y\rangle)$ for every $x,y\in \mathcal{H}$ and $a,b\in A$. Denote by $C^*(\pi,\tau)$ the $C^*$-subalgebra of $B$ generated by $\pi(A)$ and $\tau(\mathcal{H})$. 
We say a representation $(\pi,\tau)$ \emph{admits a gauge action} if there is an action $\beta$ of $\mathbb{T} = \{z\in \mathbb{C} \ : \ \lvert z\rvert = 1\}$ on $C^*(\pi,\tau)$ such that $\beta_z(\pi(a)) = \pi(a)$ and $\beta_z(\tau(\xi)) = z\tau(\xi)$ for all $a\in A$ and $\xi\in \mathcal{H}$.
\subsection{Cuntz--Pimsner algebras}
We now briefly review the construction of Cuntz--Pimsner algebras. These were first defined by Pimsner in \cite{Pimsner}. Note that \cite{Pimsner} contains  the proofs of Theorems \ref{Toeplitz-Properties} and \ref{Cuntz-Universal}, and Proposition \ref{useful_cuntz}. See also \cite[Section 4.6]{Brown-Ozawa} 

For a single  correspondence $\mathcal{H}$, set $\mathcal{H}^{\otimes 0} = A$ and $\mathcal{H}^{\otimes k} = \mathcal{H}\otimes \cdots \otimes \mathcal{H}$. The \emph{full Fock space} and \emph{$p^{\text{th}}$-cutoff Fock space} over $\mathcal{H}$ are the correspondences defined by
\begin{equation*}
\mathcal{F}(\mathcal{H}) = \bigoplus_{k = 0}^{\infty}\mathcal{H}^{\otimes k} \ \ \ \text{and} \ \ \ \mathcal{F}_p(\mathcal{H}) = \bigoplus_{k=0}^{p-1}\mathcal{H}^{\otimes k}. 
\end{equation*}
We say an elementary tensor $x = x_1\otimes\cdots\otimes x_k \in \mathcal{H}^{\otimes k}\subset\mathcal{F}(\mathcal{H})$ has \emph{length $k$} and write $\lvert x\rvert = k$.  
When $\mathcal H$ is a free module with orthonormal basis $\{\xi_i\}_{i \in I}$, we set $W$ equal to the set of all elementary tensors in the $\xi_i$, and
\begin{equation}
\label{W_kDef}
 W_k = \{\mu \in W \mid |\mu| = k\}, \quad W_{<p} = \{\mu \in W \mid |\mu| < p\}.
\end{equation}
The left action of $A$ on $\mathcal{F}(\mathcal{H})$ is given by $a\ldotp (x_1\otimes\cdots\otimes x_k) = (a\ldotp x_1)\otimes \cdots \otimes x_k$. Moreover for each $x \in \mathcal{H}$ we define the \emph{creation} operator $T_x\in \mathbb{B}(\mathcal{F}(\mathcal{H}))$ by $T_x(a) = x\ldotp a$ and $T_x(x_1\otimes \cdots\otimes x_n) = x\otimes x_1\otimes\cdots\otimes x_n$. If $a\in\mathcal{H}^{\otimes 0}$ set $T_a = a$, and if $x = x_1\otimes\cdots\otimes x_k\in \mathcal{H}^{\otimes k}$ set $T_x = T_{x_1}\cdots T_{x_k}$. Note that for an elementary tensor $y$,
\[ T_x^*(y) = \begin{cases} \langle x,y'\rangle. y'', \quad &y=y' \otimes y'', |y|=|x|; \\ 0, \quad &|y|<|x|. \end{cases} \]
Using rank-one operators $e_{x,y}\in \mathbb{K}(\mathcal{H}^{\otimes \lvert x\rvert}, \mathcal{H}^{\otimes \lvert y\rvert})$ associated to elementary tensors in $\mathcal{F}(\mathcal{H})$, we have the identity
\begin{equation}\label{Strict-Expansion}
T_xT_y^* = \sum_{k = 0}^{\infty}e_{x,y}\otimes 1_{\mathcal{H}^{\otimes k}},
\end{equation}
where convergence is understood to be strict. 

\begin{definition}\label{Toeplitz-Pimsner Algebra}
Let $\mathcal{H}$ be a correspondence over $A$. The \emph{Toeplitz--Pimsner algebra} $\mathcal{T}(\mathcal{H})$ is the $C^*$-subalgebra of $\mathbb{B}(\mathcal{F}(\mathcal{H}))$ generated by $A$ and $\{T_x \ | \ x\in \mathcal{H}\}$. 
\end{definition}

\begin{thm}\label{Toeplitz-Properties}
Let $\mathcal{T}(\mathcal{H})$ be the Toeplitz--Pimsner algebra of a correspondence $\mathcal{H}$ over $A$. 
\begin{enumerate}
\item For every $\alpha\in \mathbb{C}$, $x,y\in \mathcal{H}$, and $a,b\in A$, the creation operators satisfy
\begin{equation}
\label{Tequations}
T_{\alpha x + y} = \alpha T_x + T_y, \ \ \ T_{a\ldotp x\ldotp b} = aT_xb, \ \ \ T_x^*T_y = \langle x,y \rangle.
\end{equation}
In particular,
\begin{equation}
\label{Tspan}
\mathcal{T}(\mathcal{H}) = \overline{\mathrm{span}\{T_xT_y^* \mid x,y \in \mathcal F(\mathcal H) \text{ elementary tensors}\}}.
\end{equation}
\item (\cite{Fowler-Muhly-Raeburn}) 
The Toeplitz--Pimsner algebra $\mathcal T(\mathcal H)$ is the unique $C^*$-algebra (up to isomorphism) generated by a representation $(\pi,\tau)$ of $\mathcal H$, such that 
\begin{equation*}
\pi(A) \cap \overline{\mathrm{span}\{\tau(x)\tau(y)^* \mid x,y \in \mathcal H\}} = 0,
\end{equation*}
and which admits a gauge action.
\end{enumerate} 
\end{thm}

The Toeplitz--Pimsner algebra is too large for many purposes, so we define the Cuntz--Pimsner algebra $\mathcal{O}(\mathcal{H})$ to be a natural quotient of $\mathcal{T}(\mathcal{H})$. Denote by $J(\mathcal{H})$ the $C^*$-subalgebra of $\mathbb{B}(\mathcal{F}(\mathcal{H}))$ generated by
\begin{equation*}
\bigcup_{n=0}^\infty \mathbb{B}(\bigoplus_{k=0}^n \mathcal{H}^{\otimes k}).
\end{equation*}
The multiplier algebra $\mathcal{M}(J(\mathcal{H}))$ can be identified with all $T\in \mathbb{B}(\mathcal{F}(\mathcal{H}))$ satisfying both $TJ(\mathcal{H}) \subset J(\mathcal{H})$ and $J(\mathcal{H})T\subset J(\mathcal{H})$. In particular, there is there is an inclusion $\mathcal{T}(\mathcal{H}) \subset \mathcal{M}(J(\mathcal{H}))$. (However, note that $\mathcal{T}(\mathcal{H})$ may not contain $J(\mathcal H)$.)

\begin{definition}
\label{CPalgebra}
The \emph{Cuntz--Pimsner algebra} $\mathcal{O}(\mathcal{H})$ is the $C^*$-algebra $Q(\mathcal{T}(\mathcal{H}))$, where $Q:\mathcal{M}(J(\mathcal{H})) \rightarrow \mathcal{M}(J(\mathcal{H}))/J(\mathcal{H})$ is the quotient map. We denote by $S_x$ the image of the creation operator $T_x$ under $Q$. 
\end{definition}

Here is another description of $\mathcal{O}(\mathcal{H})$. Let $I_{\mathcal{H}} = A\cap \mathbb{K}(\mathcal{H}) \subset \mathbb{B}(\mathcal{H})$. Since $I_{\mathcal{H}}$ is an ideal in $A$ and $\mathcal{F}(\mathcal{H})I_{\mathcal{H}}$ is a $\mathbb{B}(\mathcal{F}(\mathcal{H}))$-invariant subcorrespondence of $\mathcal{F}(\mathcal{H})$, we can conclude $\mathbb{K}(\mathcal{F}(\mathcal{H})I_{\mathcal{H}}) = \overline{\mathrm{span}}\{e_{x,y} \ | \ x,y\in \mathcal{F (\mathcal{H})I_{\mathcal{H}}\}}$ is an ideal in $\mathbb{B}(\mathcal{F}(\mathcal{H}))$. 

\begin{prop}\label{useful_cuntz}
$\mathbb{K}(\mathcal{F}(\mathcal{H})I_{\mathcal{H}}) \subset\mathcal{T}(\mathcal{H})$, and in particular $\mathbb{K}(\mathcal{F}(\mathcal{H})I_{\mathcal{H}}) = \ker Q|_{\mathcal{T}(\mathcal{H})}$. In other words, 
\begin{equation*}
\mathcal{O}(\mathcal{H}) \cong \mathcal{T}(\mathcal{H})/\mathbb{K}(\mathcal{F}(\mathcal{H})I_{\mathcal{H}}). 
\end{equation*}
\end{prop}

\begin{thm}\label{Cuntz-Universal}
Let $\mathcal{O}(\mathcal{H})$ be the Cuntz--Pimsner algebra of a correspondence $\mathcal{H}$ over $A$. 
For every $\alpha \in \mathbb{C},x,y\in \mathcal{H}$, and $a,b\in A$, the following identity holds:
\begin{equation*}
S_{\alpha x + y} = \alpha S_x + S_y, \ \ \ S_{a\ldotp x\ldotp b} = aS_xb, \ \ \ S_x^*S_y = \langle x,y\rangle. 
\end{equation*}
\end{thm}

\begin{rem}\label{Cuntz=Pimsner}
If $\mathcal{H}$ is a finitely generated projective  correspondence, then $\mathbb{K}(\mathcal{H}) = \mathbb{B}(\mathcal{H})$ (see \cite{Wegge-Olsen}) and hence $A\cap \mathbb{K}(\mathcal{H}) = A$. Since $A$ is unital, $\mathbb{K}(\mathcal{F}(\mathcal{H})I_{\mathcal{H}}) = \mathbb{K}(\mathcal{F}(\mathcal{H}))$. This shows $\mathcal{T}(\mathcal{H})$ contains all of $\mathbb{K}(\mathcal{F}(\mathcal{H}))$ and $\mathcal O(\mathcal H) \cong \mathcal T(\mathcal H)/\mathbb K(\mathcal F(\mathcal H))$.  At the other extreme, if $A\cap\mathbb{K}(\mathcal{H}) = \{0\}$, the kernel of $Q$ is trivial and hence there is a *-isomorphism $\mathcal{O}(\mathcal{H}) \rightarrow \mathcal{T}(\mathcal{H})$ sending $a$ to $a$ and $S_x$ to $T_x$. 
\end{rem}

%
\section{Rokhlin dimension for $C^*$-correspondences}\label{Rokhlin dimension for correspondences}
Here is our definition of Rokhlin dimension for $C^*$-correspondences.

\begin{definition}\label{Rokhlin-Dimension-Def}
Let $A$ be a separable $C^*$-algebra and let $\mathcal{H}$ be a countably generated correspondence over $A$.
We say that $\mathcal{H}$ has \emph{Rokhlin dimension at most $d$}, $\dim_{\mathrm{Rok}}(\mathcal{H}) \le d$, if any $p\in \mathbb{N}$, there exist positive contractions 
\begin{equation*}
\{f_k^l\}_{l=0,\dots,d;\, k \in \mathbb Z/p} \subset A_\infty \cap A'
\end{equation*}
satisfying 
\begin{enumerate}
\item $f_k^lf_{k'}^l=0$ for all $l$ and all $k \ne k'$,
\item $\sum_{k,l}f_k^l = 1$, and
\item $z\ldotp f_k^l = f_{k+1}^l\ldotp z$ in $\mathcal H_\infty$, for all $k$, $l$, and $z \in \mathcal H$.
\end{enumerate}
\end{definition}

\begin{rem}\label{Rokhlin-Def-Rem}
(i)
One can of course reformulate this definition without using the sequence algebra $A_\infty$ and the sequence correspondence $\mathcal H_\infty$.
Namely, we have $\dim_{\mathrm{Rok}}(\mathcal H) \leq d$ if and only if, for any $\epsilon > 0$, any $p\in \mathbb{N}$, any finite set $F\subset A$, and any finite set $\mathcal{V} \subset \mathcal{H}$, there exist positive contractions 
\begin{equation*}
\{f_k^l\}_{l=0,\dots,d;\, k \in \mathbb Z/p} \subset A
\end{equation*}
satisfying 
\begin{enumerate}
\item $\lVert f_k^lf_{k'}^l\rVert < \epsilon$ when $k \ne k'$ and all $l$. 
\item $\lVert \sum_{k,l}f_k^l - 1\rVert < \epsilon$. 
\item $\lVert z\ldotp f_k^l - f_{k+1}^l\ldotp z\rVert < \epsilon$ for all $k$, $l$, and $z \in \mathcal{V}$.
\item $\lVert [f_k^l,a]\rVert < \epsilon$ for all $k,l$ and $a\in F$. 
\end{enumerate}
(ii)
If $A$ is a $C^*$-algebra and $\alpha \in \text{Aut}(A)$, one may define a correspondence $A^\alpha$ over $A$ as the singly generated $C^*$-module $A$ with with the left action given by $a\ldotp b = \alpha(a)b$.
There is a canonical isomorphism $\mathcal O(A^\alpha) \cong A \rtimes_\alpha \mathbb Z$ (this isomorphism fixes $A$ and sends $S_1$ to the canonical unitary of the crossed product).
Our definition of Rokhlin dimension of $C^*$-correspondences is designed to (almost) coincide with the Rokhlin dimension of $\alpha$ (with single towers) as defined by Hirshberg, Winter, and Zacharias in \cite[Definition 2.3(c)]{Hirshberg-Winter-Zacharias}.
Specifically, we ask for single towers all of height $p+1$ (whereas their definition of $\dim_{\mathrm{Rok}}^s(A,\alpha)\leq d$ asks that each colour has a single tower of height either $p$ or $p+1$).
In \cite[Proposition 2.8 and Remark 2.9]{Hirshberg-Winter-Zacharias}, they show that, up to a possible factor of $2$, their Rokhlin dimension coincides with the version with single towers all of height $p+1$.
The same argument applies to variants on the definition of Rokhlin dimension for correspondences, and as such, we have chosen to work with the simplest version of Rokhlin dimension.

(iii) 
We can simultaneously express condition (3) and the requirement that the Rokhlin contractions commute with $A$ by asking that, for any elementary tensor $z \in \mathcal F(\mathcal H)$,
\[ z\ldotp f_k^l = f_{k+\lvert z\rvert}^l\ldotp z. \]
\end{rem}

\section{Nuclear dimension of Toeplitz--Pimsner and Cuntz--Pimsner algebras}\label{Heavy lifting}
In \cite[Theorem 4.1]{Hirshberg-Winter-Zacharias}, it is shown that for a $C^*$-algebra $A$ with finite nuclear dimension and an automorphism $\alpha:A \to A$ of finite Rokhlin dimension, the crossed product $A \rtimes_\alpha \mathbb Z$ has finite nuclear dimension.
In this section we generalize this result to correspondences of finite Rokhlin dimension (this is truly a generalization, see Remark \ref{Rokhlin-Def-Rem} (i)), subject to a technical condition which is satisfied, for example, by correspondences which are finitely generated and projective as Hilbert $C^*$-modules.

One can recast the argument used to prove \cite[Theorem 4.1]{Hirshberg-Winter-Zacharias} in terms of the Fock-space representation of $\mathcal{T}(A^\alpha)$, where $A^\alpha$ is as in Remark \ref{Rokhlin-Def-Rem} (i).
Specifically, the argument makes use of outgoing maps from $\mathcal T(A^\alpha)$ to a compression of the Fock space representation; the range of these maps land in a subalgebra of $\mathbb B(\mathcal F(A^\alpha))$ which, in this case, is isomorphic to some $M_n(A)$.
It is possible to then define incoming maps $M_n(A) \to \mathcal T(A^\alpha)$, using compressions by row vectors corresponding to Rokhlin towers.
Our argument for general correspondences is based on this outline; however, there are two significant technical differences.
First, the codomain of the outgoing map will generally not be a matrix algebra over $A$, and so further input is needed to get a (uniform) bound on its nuclear dimension.
Second, there need to be suitable replacements for the row vector compressions used to construct the incoming maps.
In this section, we deal with these technicalities.

\subsection{A compressed Fock space representation and nuclear dimension}
Let $\mathcal{H}$ be a countably generated correspondence over $A$. For each $p \in \mathbb{N}$, form
\[ \mathcal{F}_p(\mathcal{H}) = \bigoplus_{k=0}^{p-1}\mathcal{H}^{\otimes k}. \]
Set
\begin{align*}
D_p(\mathcal{H}) = \mathrm{span}\{e_{x,y} \otimes 1_{\mathcal H}^{\otimes k} \mid & x,y \in \mathcal F_p(\mathcal H) \text{ elementary tensors}, \\ &\quad \max\{|x|,|y|\}+k<p\}.
\end{align*}
This is a $C^*$-subalgebra of $\mathbb B(\mathcal F_p(\mathcal H))$.
Evidently, $\mathbb{K}(\mathcal{F}_p(\mathcal{H})) $ is an ideal of $D_p(\mathcal{H})$. 


\begin{lem}\label{quasi_hypothesis}
Suppose that for every $p\in \mathbb{N}$, there is an approximate unit consisting of projections in $\mathbb{K}(\mathcal{F}_p(\mathcal{H}))$ that are quasicentral in $D_p(\mathcal{H})$. Then for every $p\in \mathbb{N}$ we have
\begin{equation*}
\dim_{\mathrm{nuc}}(D_p(\mathcal{H})) \le \dim_{\mathrm{nuc}}(A). 
\end{equation*}
\end{lem}

\begin{proof}
By induction in $p$.
It is not hard to see that $D_1(\mathcal{H}) \cong A$. Suppose the result holds for $D_{i}$ for $i = 1,\ldots,p$. Set 
\begin{equation*}
\tilde{D}_p = \overline{\mathrm{span}}\{(e_{x,y}\otimes 1_{\mathcal{H}^{\otimes k}})\otimes 1_{\mathcal{H}} \ : 0 \le \max\{\lvert x\rvert ,\lvert y\rvert\} + k < p\} \subset D_{p+1}(\mathcal{H}). 
\end{equation*}
It's clear that $\tilde{D}_p\cong D_p(\mathcal{H})$ and that
\begin{equation*}
D_{p+1}(\mathcal{H}) = C\text{*}(\mathbb{K}(\mathcal{F}_{p+1}(\mathcal{H})) \cup \tilde{D}_p) \ \ \ \text{and} \ \ \ \mathbb{K}(\mathcal{F}_{p+1}(\mathcal{H}))\cap\tilde{D}_p = \{0\}.
\end{equation*}
This shows that $D_{p+1}(\mathcal{H})$ is an extension of (an algebra isomorphic to) $D_p(\mathcal{H})$ by the compacts $\mathbb{K}(\mathcal{F}_{p+1}(\mathcal{H}))$.
The result follows by induction, Proposition \ref{propextension}, and 
By the inductive hypothesis and \eqref{Compact-Dimension}, both of these algebras have nuclear dimension at most $\dim_{\mathrm{nuc}}(A)$.
By the hypotheses of this lemma, we may apply Proposition \ref{propextension} to conclude that the nuclear dimension of $D_{p+1}(\mathcal H)$ is also at most $\dim_{\mathrm{nuc}}(A)$.
\end{proof}

\begin{example}
\label{QD-examples}
(i)
If $\mathcal{H}$ is a finitely generated projective correspondence, then so is $\mathcal{H}^{\otimes k}$ for any $k \ge 0$ (see \cite[Proposition 4.7]{Lance}). Hence, $\mathbb{K}(\mathcal{F}_p(\mathcal{H}))$ is unital for any $p\in \mathbb{N}$ and we get 
\begin{equation*}
\mathbb{K}(\mathcal{F}_p(\mathcal{H})) = D_p(\mathcal{H}) = \mathbb{B}(\mathcal{F}_p(\mathcal{H})).
\end{equation*}
Thus, such $\mathcal H$ does satisfy the condition of Lemma \ref{quasi_hypothesis} (i.e., for every $p$, there is an approximate unit of projections in $\mathbb K(\mathcal F_p(\mathcal H))$ which is quasicentral in $D_p(\mathcal H)$).

(ii)
Let $\mathcal{K}$ be countably generated free Hilbert $A$-module with orthonormal basis $(\xi_i)_{i=1}^\infty$.
Let $\{\alpha_i\}_{i=1}^\infty \subset \mathrm{Aut}(A)$ be a sequence of automorphisms of $A$, and define a correspondence by the left action $a\ldotp \xi_i = \xi_i\ldotp \alpha_i(a)$ for $i \in \mathbb N$.
This correspondence also satisfies the hypothesis of Lemma \ref{quasi_hypothesis}.

\begin{proof}
For $\mu = \xi_1\cdots \xi_k \in W_k$, set 
\[ \alpha_{\mu} = \alpha_{k} \circ \cdots \circ \alpha_{1} \in \mathrm{Aut}(A), \]
and note that $a.\mu = \mu.\alpha_\mu(a)$.

Let $W_{<p}^n$ denote the finite subset of $W_{<p}$ consisting of elementary tensors of length $<p$ in $\xi_1,\dots,\xi_n$.
Define the finite rank projection $q_n = \sum_{\zeta \in W_{<p}^n}e_{\zeta,\zeta}$.
This projection is quasicentral in $D_p(\mathcal{K})$, since for $\mu,\nu \in W$ and $k \in \mathbb N$, for sufficiently large $n$ (namely, when $n$ is greater than all indices of basis elements appearing in $\mu$ and $\nu$), we have
\begin{align*}
q_n(e_{\mu\ldotp a,\nu}\otimes 1_{\mathcal{K}^{\otimes k}}) & = 
\sum_{\zeta \in W_{<p}^n}e_{\zeta,\zeta}\sum_{\eta \in W_k} e_{\mu.a\otimes \eta, \nu\otimes \eta} \\
&= \sum_{\zeta \in W_{<p}^n}e_{\zeta,\zeta}\sum_{\eta\in W_k}e_{\mu\otimes \eta\ldotp \alpha_{\eta}(a),\nu\otimes \eta} \\
& = \sum_{\zeta\in W_{<p}^n}\sum_{\eta\in W_k}e_{\zeta\ldotp \langle\zeta, \mu\otimes \eta\rangle\alpha_{\eta}(a), \nu\otimes \eta}\\
& = \sum_{\eta\in W_k^n}e_{\mu\otimes \eta\ldotp \alpha_{\eta}(a),\nu\otimes \eta}\\
& = \sum_{\eta\in W_k}\sum_{\zeta\in W_{<p}^n}e_{\mu\otimes \eta\ldotp \alpha_{\eta}(a), \zeta\ldotp \langle\zeta, \nu\otimes \eta\rangle}\\ 
& = \sum_{\eta\in W_k}e_{\mu\otimes \eta\ldotp \alpha_{\eta}(a), \nu\otimes \eta}\sum_{\zeta\in W_{<p}^n}e_{\zeta,\zeta}\\ 
& = e_{\mu\ldotp a,\nu}\otimes 1_{\mathcal{K}^{\otimes k}}q_n. 
\end{align*}
\end{proof}
\end{example}
\subsection{Incoming maps}

This section works extensively with countably generated free correspondences: first we use the structure of such a correspondence $\mathcal K$ to define incoming maps $D_p(\mathcal K) \to \mathcal T(\mathcal K)$, and then we use Kasparov's stabilization theorem to apply such incoming maps to general countably generated correspondences.
Throughout, $\mathcal K$ will generally be a countably generated free correspondence, while $\mathcal H$ will be an arbitrary (countably generated) Hilbert module or correspondence.

For a countably generated free correspondence $\mathcal K$ over $A$ with orthonormal basis $(\xi_i)_{i=1}^\infty$, recall that $W \subset \mathcal F(\mathcal K)$ denotes the set of elementary tensors in this generating set, while (from \eqref{W_kDef}) $W_k$ and $W_{<p}$ denote the subsets of elementary tensors of length $k$ and $<p$ respectively.
As Hilbert $A$-modules, we have
\begin{equation*}
\mathcal{K}^{\otimes k}\cong \bigoplus_{W_k}A, \ \ \ \mathcal{F}_p(\mathcal{K}) \cong \bigoplus_{W_{<p}}A, \ \ \ \mathcal{F}(\mathcal{K}) \cong \bigoplus_WA.
\end{equation*}

\subsubsection{Free correspondences}\label{technical0}

\begin{lem}\label{i-love-lamp}
Let $A$ be a $C^*$-algebra, let $\mathscr{H}$ be a Hilbert $A$-module, and let $I$ be an index set.
There is an inclusion $\mathbb{B}(\bigoplus_IA) \hookrightarrow \mathbb{B}(\bigoplus_I \mathscr{H})$. More specifically, an operator $S\in \mathbb{B}(\bigoplus_IA)$ acts adjointably on $\bigoplus_I\mathscr{H}$ via
\begin{equation}
\label{i-love-lamp-eq}
(x_i)_{i \in I} \mapsto (\sum_{i\in I} (\langle \xi_j,S\xi_i\rangle \ldotp x_i)_{j \in J}, 
\end{equation}
where $(\xi_i)_{i\in I}$ is the canonical orthonormal basis for $\bigoplus_I A$.
\end{lem}

\begin{proof}
The map $a\otimes x \mapsto a\ldotp x$ implements an isomorphism between $A\otimes \mathscr{H}$ and $\mathscr{H}$.
Let us show that the map $m:  (a_i)_{i\in I} \otimes x\mapsto (a_i \ldotp x)_{j \in I}$ extends to an isomorphism between $(\bigoplus_IA)\otimes \mathscr{H}$ and $\bigoplus_I\mathscr{H}$.
We have
\begin{eqnarray*}
\langle (a_i)_i \otimes x, (b_i)_i \otimes y\rangle  
&\stackrel{\eqref{TensorIP}}=& \langle x, \langle (a_i)_i,(b_i)_i\rangle\ldotp y\rangle \\
&=& \langle x, \sum_i a_i^*b_i \ldotp y \rangle \\
&=& \sum_i \langle a_i\ldotp x, b_i\ldotp y \rangle \\
&=& \langle (a_i\ldotp x)_i, (b_i\ldotp y)_i \rangle \\
&=& \langle m((a_i)_i \otimes x), m((b_i)_i \otimes y) \rangle
\end{eqnarray*}
so that $m$ can be extended, as a $\langle \cdot,\cdot\rangle$-preserving linear functional, to the tensor product $(\bigoplus_IA)\otimes \mathscr{H}$. Moreover, the image of $m$ contains the dense set consisting of finitely supported elements in $\bigoplus_I\mathscr{H}$, so that $m$ is a unitary operator.

The result follows since $T\mapsto mTm^{-1}$ is a *-isomorphism between $\mathbb{B}((\bigoplus_IA)\otimes \mathscr{H})$ and $\mathbb{B}(\bigoplus_I\mathscr{H})$, and composing this with the natural embedding $\mathbb{B}(\bigoplus_IA) \hookrightarrow\mathbb{B}((\bigoplus_IA)\otimes \mathscr{H})$ yields a map satisfying \eqref{i-love-lamp-eq}. 
\end{proof}

\begin{cor}\label{techcortwo}
Let $A$ be a $C^*$-algebra, let $\mathcal K$ be a countably generated free correspondence with orthonormal basis $\{\xi_i\}_{i\in I}$, and define $W_{<p}$ by \eqref{W_kDef}.
For each $p \in \mathbb{N}$, there is an inclusion $\mathbb{B}(\mathcal{F}_p(\mathcal{K}))  \hookrightarrow \mathbb{B}(\bigoplus_{W_{<p}}\mathcal{F}(\mathcal{K}))$ that sends $e_{\mu\ldotp a,\nu}\otimes 1_{\mathcal{K}^{\otimes k}}$ (where $\mu,\nu \in W$, $a \in A$, and $\max\{|\mu|,|\nu|\} + k < p$) 
\begin{equation}
\label{techcortwo-eq1}
(x_{\zeta})_{\zeta \in W_{<p}} \mapsto \sum_{\eta,\eta'\in W_k} (\delta_{\zeta',\mu\otimes \eta'} \langle \eta', a\ldotp \eta\rangle \ldotp x_{\nu\otimes \eta})_{\zeta'' \in W_{<p}},
\end{equation}
and $a$ (as an operator in $\mathbb{B}(\mathcal{F}_p(\mathcal{K}))$) to the operator
\begin{equation}
\label{techcortwo-eq2}
(x_{\zeta})_{\zeta \in W_{<p}} \mapsto \sum_{\zeta \in W_{<p}} (\langle \zeta',a\ldotp \zeta \rangle\ldotp x_{\zeta})_{\zeta' \in W_{<p}}.
\end{equation}
\end{cor}

\begin{proof}
Applying Lemma \ref{i-love-lamp} with $\mathcal H=\mathcal F(\mathcal K)$ and $I=W_{<p}$ yields an embedding $\iota:\mathbb B(\bigoplus_{W_{<p}} A) \to \mathbb B(\bigoplus_{W_{<p}} \mathcal F(\mathcal K))$ such that for $S \in \mathbb B(\bigoplus_{W_{<p}} A)$ and $(x_{\zeta})_{\zeta \in W_{<p}}$,
\[ \iota(S)((x_\zeta)_{\zeta\in W_{<p}}) = \sum_{\zeta \in W_{<p}} (\langle \xi_{\zeta'},S\xi_\zeta \rangle\ldotp x_\zeta)_{\zeta' \in W_{<p}}. \]
We use the canonical identification of $\mathcal F_p(\mathcal K)$ with $\bigoplus_{W_{<p}} A$ to view $\iota$ as an inclusion of $\mathbb(F_p(\mathcal K))$ to $\mathbb B(\bigoplus_{W_{<p}} \mathcal F(\mathcal K))$.
This identification takes an elementary tensor $\zeta \in W_{<p}$ to the orthonormal basis element $\xi_\zeta$.
Since $a \in A$ acts on $\mathcal F_p(\mathcal K)$ by sending $\zeta \in W_{<p}$ to $a.\zeta$, \eqref{techcortwo-eq2} follows immediately.
For \eqref{techcortwo-eq1}, using \eqref{RankOneFormula} and \eqref{TensorBalance}, first note that for $\zeta \in W_{<p}$,
\[ (e_{\mu\ldotp a,\nu} \otimes 1_{\mathcal H^{\otimes k}})(\zeta) = \begin{cases} \mu \otimes a\ldotp \eta, \quad &\zeta=\nu \otimes \eta, |\eta|=k; \\ 0, \quad &\text{otherwise}. \end{cases} \]
For $\zeta = \mu \otimes \eta \in W_{<p}$ where $\eta \in W_k$ and for $\zeta' \in W_{<p}$,
\begin{align*}
\langle \zeta',(e_{\mu\ldotp a,\nu} \otimes 1_{\mathcal H^{\otimes k}})\zeta \rangle
&= \langle \zeta',\mu \otimes a\ldotp \eta \rangle \\
&= \begin{cases} \langle \eta', a\ldotp \eta \rangle, \quad &\zeta' = \mu \otimes \eta', |\eta'| = k; \\ 0, \quad &\text{otherwise}. \end{cases}
\end{align*}
Putting this together, we have
\begin{align*}
\iota(e_{\mu\ldotp a,\nu} \otimes 1_{\mathcal H^{\otimes k}})((x_\zeta)_{\zeta \in W_{<p}})
&= \sum_{\zeta \in W_{<p}} (\langle \zeta',(e_{\mu\ldotp a,\nu} \otimes 1_{\mathcal H^{\otimes k}})\zeta \rangle\ldotp x_\zeta)_{\zeta' \in W_{<p}} \\
&= \sum_{\eta,\eta' \in W_k} (\delta_{\zeta',\mu \otimes \eta'} \langle \eta', a\ldotp \eta\rangle \ldotp x_{\mu \otimes \eta})_{\zeta' \in W_{<p}},
\end{align*}
as required.
\end{proof}

\begin{lem}\label{Rows-Vectors}
Let $A$ be a $C^*$-algebra, let $\mathscr{H}$ be a Hilbert $A$-module, and let $I$ be a countable index set.
Let $\{T_i\}_{i\in I}$ be a collection of isometries in $\mathbb{B}(\mathscr{H})$  with orthogonal ranges such that $\sum_i T_iT_i^*$ converges strictly in $\mathbb{B}(\mathscr{H})$.
Then the map $[T_i]_I:\bigoplus_I\mathscr{H} \rightarrow \mathscr{H}$ given by $(x_i)_{i\in I}\mapsto \sum_{i \in I} T_ix_i$ is an adjointable operator  with adjoint given by 
\begin{equation}
\label{TiAdjoint}
[T_i]_I^*(x) = (T_i^*x)_{i \in I}.
\end{equation}
Moreover, $[T_i]_I$ is an isometry. 
\end{lem}

\begin{proof}
The hypothesis that $\sum_i T_iT_i^*$ converges strictly implies that the formula for $[T_i]_I^*(x)$ does define an element of $\bigoplus_I \mathcal H$, and therefore this formula produces a well-defined map $\mathcal H \to \bigoplus_I \mathcal H$.
For an indexed family $(x_i)_{i \in I}$ of pairwise orthogonal elements of $\mathcal H$, note that $\sum_i x_i$ converges to an element $x \in \mathcal H$ if and only if $(x_i)_{i \in I}$ represents an element $y$ of $\bigoplus_I \mathcal H$, and in this case, $\langle x,x\rangle_{\mathcal H} = \langle y,y\rangle_{\bigoplus_I \mathcal H}$.
Since the $T_i$ have orthogonal ranges, it follows from these facts that $[T_i]_I$ is a well-defined isometry.
It is an easy calculation to see that the formula for $[T_i]_I^*$ does indeed provide an adjoint to $[T_i]_I$.
\end{proof}

\begin{rem}
The assumption that $\sum_i T_iT_i^*$ converges strictly is not automatic; here is an example.
Let $A=l^\infty(\mathbb C)$ and let $\mathcal H = \bigoplus_{\mathbb N} A$; write an element of $\mathcal H$ as $(x_i^j)$ where for each $j$, $\sum_i |x_i^j|^2$ converges (and is uniformly bounded in $j$).
For each $j \in \mathbb N$, pick an injective map $\theta_j:\mathbb N \times \mathbb N \to \mathbb N$ such that $\theta_j(1,j)=1$.
For each $k \in \mathbb N$, define $T_k:\mathcal H \to \mathcal H$ by $T_k((x_i^j)) = (y_i^j)$ where
\[ y_i^j = \begin{cases} x_{i'}^j, \quad &i=\theta^j(i',k); \\ 0,\quad &\text{otherwise}. \end{cases} \]
Using the fact that the $\theta_j$ are injective, one easily computes
\[ \langle (x_i^j), (x_i^j) \rangle = (\sum_i |x_i^j|^2)_j = \langle T_k((x_i^j)), T_k(x_i^j)) \rangle, \]
so that each $T_k$ is isometric.
Injectivity of the $\theta_j$ also implies that the $T_k$ have pairwise orthogonal ranges.

To see that $\sum_k T_kT_k^*$ does not converge strictly, let us check that $\sum_k T_kT_k^* \xi_1$ does not converge (in $\mathcal H$).
Note that $\xi_1 = (\delta_{i,1})_{i,j}$, and that for $(x_i^j) \in \mathcal H$,
\[ T_kT_k^*(x_i^j) = (\chi_{A_{j,k}}(i)y_i^j), \]
where
\[ A_{j,k} = \{\theta_j(n,k) \mid n \in \mathbb N\}. \]
By our choice of $\theta_j$, we have $1 \in A_{j,k}$ if and only if $j=k$.
Therefore,
\[ \sum_{k=1}^n T_kT_k^*\xi_1 = \xi_1\ldotp \chi_{\{1,\dots,n\}} \]
(viewing $\chi_{\{1,\dots,n\}}$ as an element of $l^\infty(\mathbb N)=A$).
The sequence $(\chi_{\{1,\dots,n\}})_n$ does not converge (in norm) in $A$, so that $\sum_k T_kT_k^* \xi_1$ does not converge in $\mathcal H$.
\end{rem}

\begin{cor}\label{techcorone}
Let $A$ be a $C^*$-algebra, let $\mathcal K$ be a countably generated free correspondence with orthonormal basis $\{\xi_i\}_{i\in I}$, and define $W_{k}$ by \eqref{W_kDef}.
For any $k \ge 0$, there is an isometry $[T_{\eta}]_{W_k}  \in \mathbb{B}(\bigoplus_{W_k}\mathcal{F}(\mathcal{K}), \mathcal{F}(\mathcal{K}))$ given by 
\begin{equation*}
(x_{\eta})_{\eta \in W_k} \mapsto \sum_{\eta\in W_k}T_{\eta}( x_{\eta}).
\end{equation*}
\end{cor}

\begin{proof}
Note that $(T_\eta)_{\eta \in W_k}$ is a family of isometries with pairwise orthogonal images, and that $\sum_{\eta\in W_k}T_{\eta}T_{\eta}^*$ converges strictly to $1_{\mathcal{F}(\mathcal{K})} - 1_{\mathcal{F}_k(\mathcal{K})}$ in $\mathbb{B}(\mathcal{F}(\mathcal{K}))$.
Thus setting $I=W_k$ and $\mathscr{H} =\mathcal{F}(\mathcal{K})$, the hypotheses of Lemma \ref{Rows-Vectors} are satisfied; this lemma shows that $[T_\eta]_{W_k}$ is a well-defined isometry.
\end{proof}

\begin{rem}
For each $p \in \mathbb{N}$ and $0 \le k < p$, we can regard $[T_{\eta}]_{W_k}$ as being an element in $\mathbb{B}(\bigoplus_{W_{<p}}\mathcal{F}(\mathcal{K}), \mathcal{F}(\mathcal{K}))$ by identifying $\bigoplus_{W_{<p}}\mathcal{F}(\mathcal{K})$ with $\bigoplus_{W_0}\mathcal{F}(\mathcal{K})\oplus\cdots\oplus \bigoplus_{W_{p-1}}\mathcal{F}(\mathcal{K})$, and defining  $[T_{\eta}]_{W_k}$ to be zero on $\bigoplus_{W_{k'}}\mathcal{F}(\mathcal{K})$ for $k' \ne k$. 
\end{rem}

Recall from Definition \ref{CPalgebra} that, for $x \in \mathcal K$, $S_x$ denotes the image of $T_x$ in the Cuntz--Pimsner algebra $\mathcal O(\mathcal K)$.

\begin{lem}\label{maps}
Let $A$ be a $C^*$-algebra, let $\mathcal K$ be a countably generated free correspondence with orthonormal basis $\{\xi_i\}_{i\in I}$, and define $W_{<p}$ by \eqref{W_kDef}.
Let $p \in \mathbb{N}$ and let $\mathcal{G} =  (g_0,\ldots,g_{p-1})$ be an tuple of positive contractions in $A$.
Set 
\begin{equation}
\label{Rdef}
R_{\mathcal G} = \sum_{k=0}^{p-1}g_k[T_{\eta}]_{k} \in \mathbb{B}(\bigoplus_{W_{<p}}\mathcal{F}(\mathcal{K}), \mathcal{F}(\mathcal{K})).
\end{equation}
Define $\sigma_{\mathcal G}:D_p(\mathcal K) \to \mathbb B(\mathcal F(\mathcal K))$ to be the following composition:
\[ 
D_p(\mathcal K) \subset \mathbb B(\mathcal F_p(\mathcal K)) \to \mathbb B(\bigoplus_{W_{<p}}\mathcal F(\mathcal K)) \stackrel{R_{\mathcal G}\cdot R_{\mathcal G}^*}{\longrightarrow} \mathbb B(\mathcal F(\mathcal K)), \]
where the first map is the inclusion given by Corollary \ref{techcortwo}.
Let $Q:\mathcal M(J(\mathcal K)) \to \mathcal M(J(\mathcal K))/J(\mathcal K)$ denote the quotient map.
Then $\sigma_{\mathcal G}(D_p(\mathcal K)) \subset Q^{-1}(\mathcal O(\mathcal K))$, and so there exists a c.p.\ map
\[ \rho_{\scriptscriptstyle{\mathcal{G}}} = Q \circ \sigma_{\mathcal G}: D_p(\mathcal{K}) \rightarrow \mathcal{O}(\mathcal{K}). \]
This map satisfies
\begin{equation*}
\label{rhoMapDef}
\rho_{\scriptscriptstyle{\mathcal G}}(e_{x,y}\otimes 1_{\mathcal{K}^{\otimes k}}) = g_{k +\lvert x\rvert}S_xS_y^*g_{k + \lvert y\rvert}
\end{equation*}
for elementary tensors $x,y$ in $\mathcal{F}_p(\mathcal{K})$ and $k \in \mathbb N$ such that $\max\{|x|,|y|\}+k<p$.
\end{lem}

\begin{proof}
For $a\in A$, $\mu,\nu \in W_{<p}$, and $k \ge 0$ such that $e_{\mu\ldotp a,\nu}\otimes 1_{\mathcal{K}^{\otimes k}} \in D_p(\mathcal{K})$, and $z\in \mathcal{F}(\mathcal{K})$, we have
\begin{eqnarray*}
(R_{\mathcal G}e_{\mu\ldotp a,\nu}\otimes 1_{\mathcal{K}^{\otimes k}}R_{\mathcal G}^*)(z)
&\stackrel{\eqref{TiAdjoint}}=& R_{\mathcal G}(e_{\mu\ldotp a,\nu}\otimes 1_{\mathcal{K}^{\otimes k}})(T_{\zeta}^*g_{\lvert \zeta\rvert}\ldotp z)_{\zeta \in W_{<p}} \\ 
&\stackrel{\eqref{techcortwo-eq1}}=& R_{\mathcal G}(\sum_{\eta,\eta'\in W_k} \delta_{\zeta',\mu\otimes \eta'} \langle \eta',a\ldotp \eta\rangle \ldotp T_{\nu \otimes \eta}^* g_{\lvert \nu\otimes \eta\rvert}\ldotp z)_{\zeta'\in W_{<p}} \\
&=& \sum_{\eta,\eta'\in W_{k}} g_{\lvert \mu\otimes\eta'\rvert} T_{\mu\otimes\eta'} \langle \eta',a\ldotp \eta\rangle\ldotp T_{\nu\otimes\eta}^* g_{\lvert \nu\otimes\eta \rvert}\ldotp z \\
&=& g_{\lvert \mu\rvert + k} T_\mu \left(\sum_{\eta,\eta'\in W_{k}} T_{\eta'}\langle \eta',a\ldotp \eta\rangle T_\eta^*\right) T_\nu^* g_{\lvert \nu\rvert+k}\ldotp z.
\end{eqnarray*}
Note that for a basis element $\xi \in W$,
\begin{align*}
\sum_{\eta,\eta' \in W_k} T_{\eta'}\langle \eta',a\ldotp \eta\rangle T_\eta^* \xi
&= \begin{cases} \sum_{\eta'\in W_k} T_{\eta'} \langle \eta',a\ldotp \eta\rangle \xi', \quad &\xi=\eta\otimes \xi', \eta\in W_k; \\ 0,\quad &|\xi|<k  \end{cases} \\
&= \begin{cases} a\ldotp \eta\otimes \xi', \quad &\xi=\eta\otimes \xi', \eta\in W_k; \\ 0,\quad &|\xi|<k  \end{cases} \\
&= (1_{\mathcal F(\mathcal K)}-1_{\mathcal F_k(\mathcal K)})\xi,
\end{align*}
and by linearity, this formula continues to hold for all $\xi \in \mathcal F(\mathcal K)$.
Putting these together, we find
\[
(R_{\mathcal G}e_{\mu\ldotp a,\nu}\otimes 1_{\mathcal{K}^{\otimes k}}R_{\mathcal G}^*)(z) = (g_{\lvert\mu\rvert+k}T_{\mu}a(1_{\mathcal{F}(\mathcal{K})} - 1_{\mathcal{F}_k(\mathcal{K})})T_{\nu}^*g_{\lvert \nu\rvert}+k)(z).
\]
Define $\rho_{\mathcal G}=Q\circ\sigma_{\mathcal G}:D_p(\mathcal K) \to \mathcal{M}(J(\mathcal{K}))$, and we see that \eqref{rhoMapDef} holds for $x=\mu.a, y=\nu$ where $\mu,\nu \in W_{<p}$.
By density, \eqref{rhoMapDef} holds for all elementary tensors $x,y$, and the image of $\rho_{\mathcal G}$ is contained in $\mathcal O(\mathcal K)$ (i.e., the image of $\sigma_{\mathcal G}$ is contained in $Q^{-1}(\mathcal O(\mathcal K))$).
\end{proof}

\begin{lem}\label{moremaps}
Let $A$ be a $C^*$-algebra, let $\mathcal K$ be a countably generated free correspondence with orthonormal basis $\{\xi_i\}_{i\in I}$, and define $W_{<p}$ by \eqref{W_kDef}.
Let $\mathcal{G} = (g_0,\ldots,g_{p-1})$, let $\delta$ be the maximum value of $\lVert g_ig_j\rVert$, $i\ne j$, and define $R_{\mathcal G}$ by \eqref{Rdef}.
\begin{enumerate}
\item Regarding $g_k^2|_{\mathcal{K}^{\otimes k}}$ as an element of $\mathbb{B}(\bigoplus_{W_{<p}}\mathcal{F}(\mathcal{K}))$, we have 
\begin{equation*}
R_{\mathcal G}^*R_{\mathcal G} \approx_{p^2\delta} g_0^2|_{\mathcal{K}^{\otimes 0}} + \cdots + g_{p-1}^2|_{\mathcal{K}^{\otimes p-1}}.
\end{equation*}  
\item Let $w$ and $z$ be elementary tensors in $\mathcal{F}_p(\mathcal{K})$ satisfying $w\ldotp g_l \approx_{\delta} g_{l + \lvert w\rvert}\ldotp w$ and $z\ldotp g_l\approx_{\delta} g_{l + \lvert z\rvert}\ldotp z$. If $k \in \mathbb N$ is such that $\max\{\lvert w\rvert,\lvert z\rvert\} + k < p$, then 
\begin{equation*}
\lVert [R_{\mathcal G}^*R_{\mathcal G}, e_{w,z}\otimes 1_{\mathcal{K}^{\otimes k}}]\rVert < 2(p^2 + 2)\delta. 
\end{equation*}
\end{enumerate}
\end{lem}

\begin{proof}
(1):
For $(x_\zeta)_{\zeta\in W_{<p}} \in \bigoplus_{W_{<p}} \mathcal F(\mathcal K)$,
\begin{align*}
[T_\eta]_{W_k}^*g_k^2[T_\eta]_{W_k}(x_\zeta)_{\zeta\in W_{<p}}
&= [T_\eta]_{W_k}^*\sum_{\zeta \in W_k} g_k\cdot \zeta \otimes x_\zeta \\
&= (\chi_{W_k}(\zeta) g_k^2\cdot x_\zeta)_{\zeta \in W_{<p}},
\end{align*}
and therefore,
\begin{align*}
g_0^2|_{\mathcal{K}^{\otimes 0}} + \cdots + g_{p-1}^2|_{\mathcal{K}^{\otimes p-1}}
&= \sum_k [T_\eta]_{W_k}^*g_k^2[T_\eta]_{W_k} \\
&\approx_{p^2\delta} \sum_{k,k'} [T_\eta]_{W_k}^*g_kg_{k'}[T_\eta]_{W_{k'}} \\
&= R_{\mathcal G}^*R_{\mathcal G}.
\end{align*}
where $\zeta\in W_k$ and $x\in \mathcal{F}(\mathcal{K})$. 

(2):
Using (1) we have
\begin{align*}
(e_{w,z}\otimes 1_{\mathcal{K}^{\otimes k}})(R_{\mathcal G}^*R_{\mathcal G}) & \approx_{p^2\delta}(e_{w,z}\otimes 1_{\mathcal{K}^{\otimes k}})(g_0^2|_{\mathcal{K}^{\otimes 0}} + \cdots + g_{p-1}^2|_{\mathcal{K}^{\otimes p-1}}) \\ 
& = (e_{w,z}\otimes 1_{\mathcal{K}^{\otimes k}})(g^2_{k + \lvert z\rvert}|_{\mathcal{K}^{\otimes k + \lvert z\rvert}}) \\
& = e_{w,g^2_{k + \lvert z\rvert}\ldotp z}\otimes 1_{\mathcal{K}^{\otimes k}} \\ 
& \approx_{2\delta} e_{w,z\ldotp g_k^2}\otimes 1_{\mathcal{K}^{\otimes k}} \\ 
& = e_{w\ldotp g^2_k,z}\otimes 1_{\mathcal{K}^{\otimes k}} \\ 
&  \approx_{2\delta} e_{g^2_{k + \lvert w\rvert}\ldotp w,z}\otimes 1_{\mathcal{K}^{\otimes k}} \\ 
& = (g^2_{k + \lvert w\rvert}|_{\mathcal{K}^{\otimes k + \lvert w\rvert}})(e_{w,z}\otimes 1_{\mathcal{K}^{\otimes k}}) \\ 
& = (g^2_0|_{\mathcal{K}^{\otimes 0}} + \cdots + g^2_{p-1}|_{\mathcal{K}^{\otimes p-1}})e_{w,z}\otimes 1_{\mathcal{K}^{\otimes k}} \\ 
& \approx_{p^2\delta} (R_{\mathcal G}^*R_{\mathcal G})(e_{w,z}\otimes 1_{\mathcal{K}^{\otimes k}}). 
\end{align*}
\end{proof}
\subsubsection{Countably generated correspondences}\label{technical}

Let $\mathcal{H}$ be a countably generated correspondence over $A$. By Kasparov's Stabilization Theorem (Theorem \ref{Kasparov's Stabilization Theorem}), there is a countably generated free Hilbert $A$-module $\mathcal{H}'$ such that their direct sum $\mathcal{K} = \mathcal{H}\oplus\mathcal{H}'$ is free. Choose a left action of $A$ on $\mathcal{H}'$ such that $A\cap\mathbb{K}(\mathcal{H}') = \{0\}$ (one may take, for example, the canonical left action of $A$ on $\bigoplus_{\mathbb N} A$). The diagonal action of $A$ turns $\mathcal{K}$ into a correspondence. The orthogonal projection $P_{\mathcal{H}} \in \mathbb{B}(\mathcal{K})$ onto $\mathcal{H}\oplus 0$ commutes with the image of the left action of $A$ on $\mathcal{K}$, so for $k > 0$ the map 
\begin{equation*}
P_{\mathcal{H}}^{\otimes k} : x_1\otimes\cdots\otimes x_k \mapsto P_{\mathcal{H}}x_1\otimes \cdots \otimes P_{\mathcal{H}}x_k
\end{equation*}
is an orthogonal projection in $\mathbb{B}(\mathcal{K}^{\otimes k})$. 

\begin{lem}
Let $A$ be a $C^*$-algebra, let $\mathcal H$ be a countably generated correspondence over $A$, and let $\mathcal K$ and $P_{\mathcal H}$ be as above.
Allowing $P_{\mathcal{H}}^{\otimes 0}$ to mean the identity map on $A$, we have
\begin{enumerate}
\item $\mathcal{H}^{\otimes k} = P_{\mathcal{H}}^{\otimes k}(\mathcal{K}^{\otimes k})$,
\item $\mathcal{F}_p(\mathcal{H}) = (\sum_{k=0}^{p-1}P_{\mathcal{H}}^{\otimes k})(\mathcal{F}_p(\mathcal{K}))$, and 
\item $\mathcal{F}(\mathcal{H}) = (\sum_{k\ge 0}P_{\mathcal{H}}^{\otimes k})(\mathcal{F}(\mathcal{K}))$. 
\end{enumerate}
\end{lem}

\begin{proof}
For the first part, the map $V_k: x_1\otimes \cdots \otimes x_k \mapsto (x_1, 0)\otimes \cdots\otimes (x_k, 0)$ implements an isomorphism between the $k$-fold tensor product $\mathcal{H}^{\otimes k}$ and the image of $P_{\mathcal{H}}^{\otimes k}$ in $\mathcal{K}^{\otimes k}$. The second part follows easily from the first since the sums are finite. Lastly, it is clear that the series $\sum_{k\ge 0}P_{\mathcal{H}}^{\otimes k}$ converges strictly to a projection in $\mathbb{B}(\mathcal{F}(\mathcal{K}))$ with the desired property. 
\end{proof}

Denote the projections $\sum_{k=0}^{p-1}P_{\mathcal{H}}^{\otimes k}$ and $\sum_{k\ge 0}P_{\mathcal{H}}^{\otimes k}$ given above by $P_{\mathcal{F}_p(\mathcal{H})}$ and $P_{\mathcal{F}(\mathcal{H})}$, respectively. If $z = (x_1,y_1)\otimes\cdots\otimes (x_k,y_k)\in \mathcal{K}^{\otimes k}$ is an elementary tensor, we say that $z$ is an \emph{$\mathcal{H}$-elementary tensor} if $y_1 = \cdots = y_k = 0$. Equivalently, $z$ is $\mathcal{H}$-elementary if $P_{\mathcal{H}}^{\otimes k}(z) = z$. 

\begin{prop}\label{sumuniversecuntz}
Let $(\pi,\tau)$ be the representation of $\mathcal{H}$ in $\mathcal{O}(\mathcal{K})$ given by $\pi: a\mapsto a$ and $\tau: x\mapsto S_{(x,0)}$. Then there is a *-isomorphism $\theta: C^*(\pi,\tau) \rightarrow \mathcal{T}(\mathcal{H})$ sending $a$ to $a$ and $S_{(x,0)}$ to $T_x$. 
\end{prop}

\begin{proof}
Consider first the representation $(\tilde{\pi},\tilde{\tau})$ of $\mathcal{H}$ in $\mathcal{T}(\mathcal{K})$ given by $\tilde{\pi}: a\mapsto a$ and $\tilde{\tau}: x\mapsto T_{(x,0)}$. It is clear that 
\begin{equation*}
\tilde{\pi}(A)\cap \overline{\mathrm{span}}\{\tilde{\tau}(x)\tilde{\tau}(y)^* \ : \ x,y\in \mathcal{H}\} = \{0\}.
\end{equation*}
If we restrict the gauge action on $\mathcal{T}(\mathcal{K})$ to $C^*(\pi,\tau)$, Theorem \ref{Toeplitz-Properties} (2) gives a *-isomorphism $\tilde{\theta}: C^*(\tilde{\pi},\tilde{\tau}) \rightarrow \mathcal{T}(\mathcal{H})$ sending $a$ to $a$ and $T_{(x,0)}$ to $T_x$. Now, since the left action of $A$ on $\mathcal{H}'$ was defined so that $A\cap \mathbb{K}(\mathcal{H}') = \{0\}$, each $a\in A\subset \mathbb{B}(\mathcal{K})$ differs from any finite sum of rank-one operators by at least $\lVert a\rVert$. Therefore, $A\cap \mathbb{K}(\mathcal{K}) = \{0\}$ and so by Remark \ref{Cuntz=Pimsner} there is a *-isomorphism $\bar{\theta}: \mathcal{O}(\mathcal{K}) \rightarrow \mathcal{T}(\mathcal{K})$ sending $a$ to $a$ and $S_{(x,y)}$ to $T_{(x,y)}$. Moreover, it's clear that $\bar{\theta}(C^*(\pi,\tau)) = C^*(\tilde{\pi},\tilde{\tau})$ so taking $\theta = \tilde{\theta}\circ \bar{\theta}|_{\scriptscriptstyle{C^*(\pi,\tau)}}$ completes the proof. Here is a diagram illustrating the maps involved.

\begin{center}
\begin{tikzpicture}[descr/.style={fill=white,inner sep=2.5pt}]
\matrix (m) [matrix of math nodes, row sep=4em, column sep=9em]
{
(A,\mathcal{H})  & C^*(\pi,\tau) \subset \mathcal{O}(\mathcal{K}) \\
\mathcal{T}(\mathcal{H}) & C^*(\tilde{\pi},\tilde{\tau})\subset \mathcal{T}(\mathcal{K})\\ 
};
\path[->]
(m-1-1) edge node[auto] {$ (\pi,\tau)$} (m-1-2)
(m-1-1) edge node[auto,swap] {$(a\mapsto a,x\mapsto T_x)$} (m-2-1)
(m-1-2) edge node[auto] {$\bar{\theta}|_{\scriptscriptstyle{C^*(\pi,\tau)}}$} (m-2-2)
(m-2-2) edge node[auto] {$\tilde{\theta}$} (m-2-1)
(m-1-1) edge node[descr] {$ (\tilde{\pi},\tilde{\tau}) $} (m-2-2);
\end{tikzpicture}
\end{center}
\end{proof}

\begin{lem}\label{embed}
Let $A$ be a $C^*$-algebra, let $\mathcal H$ be a countably generated correspondence over $A$, and let $\mathcal K$ be as above.
Let $D_p^0(\mathcal{K})$ be the $C^*$-subalgebra of $D_p(\mathcal{K})$ generated by elements of the form $e_{x,y}\otimes 1_{\mathcal{K}^{\otimes k}}$ where $x$ and $y$ are $\mathcal{H}$-elementary tensors. Then the map
\begin{equation*}
\gamma: e_{x_1\otimes\cdots\otimes x_k,y_1\otimes\cdots\otimes y_l}\mapsto e_{(x_1,0)\otimes\cdots\otimes (x_k,0), (y_1,0)\otimes\cdots\otimes(y_l,0)}
\end{equation*}
extends to an isomorphism $\gamma: D_p(\mathcal{H}) \rightarrow D_p^0(\mathcal{K})$. 
In particular if for each $p$, there is an approximate unit of projections in $\mathbb K(\mathcal F_p(\mathcal H))$ which is quasicentral in $D_p(\mathcal H)$, then $\dim_{\mathrm{nuc}}(D_p^0(\mathcal{K})) \le \dim_{\mathrm{nuc}}(A)$. 
\end{lem}

\begin{proof}
The map $\bar{V} = \sum_{k=0}^{p-1}V_k$ sending $a$ to $a$ and $x_1\otimes\cdots\otimes x_k$ to $(x_1,0)\otimes \cdots\otimes (x_k,0)$ extends to a unitary from $\mathcal{F}_p(\mathcal{H})$ to $P_{\mathcal{F}_p(\mathcal{H})}(\mathcal{F}_p(\mathcal{K}))$. Let $l,m,k \ge 0$ satisfy $0 \le \max\{l,m\} + k < p$ and let $x,y,w,z \in\mathcal{F}_p(\mathcal{H})$ be elementary tensors of lengths $l,m,m,$ and $k$, respectively. We have 
\begin{align*}
\bar{V}^*e_{\bar{V}x,\bar{V}y}\otimes 1_{\mathcal{K}^{\otimes k}}\bar{V}(w\otimes z) & = \bar{V}^*(\bar{V}x\ldotp \langle \bar{V}y,\bar{V}w \rangle\otimes \bar{V}z)\\
& = \bar{V}^*(\bar{V}x\ldotp \langle y,w\rangle \otimes \bar{V}z) \\ 
& = e_{x,y}\otimes 1_{\mathcal{H}^{\otimes k}}(w\otimes z),
\end{align*}
so $\gamma$ is implemented by unitary conjugation and the result follows. 
\end{proof}

\begin{cor}\label{spock_cor}
Let $A$ be a $C^*$-algebra, let $\mathcal H$ be a countably generated correspondence over $A$, and let $\mathcal K$ be as above.
Let $\mathcal{G} = (g_0,\ldots,g_{p-1})$ be a tuple of positive contractions in $A$.
There is a c.p.\ map $\bar\rho_{\mathcal G}:D_p(\mathcal{H}) \rightarrow \mathcal{T}(\mathcal{H})$ such that, for elementary tensors $x,y$ in $\mathcal{F}_p(\mathcal{H})$ and $k$ such that $\max\{\lvert x\rvert,\lvert y\rvert\}+k<p$,
\begin{equation}
\label{barrhoeq}
\bar\rho_{\mathcal G}(e_{x,y}\otimes 1_{\mathcal{H}^{\otimes k}}) = g_{k + \lvert x\rvert}T_xT_y^*g_{k+\lvert y\rvert},
\end{equation}
and the following commutes
\begin{center}
\begin{tikzpicture}
\matrix (m) [matrix of math nodes, row sep=2em,column sep=8em, text height=1.5ex, text depth=0.25ex]
{
D_p(\mathcal{H}) & D_p^0(\mathcal{K})\subset D_p(\mathcal{K}) \\
\mathcal{T}(\mathcal{H}) & C^*(\pi,\tau)\subset \mathcal{O}(\mathcal{K}) \\
};
\path[->,font=\small]
(m-1-1) edgenode[auto] {$\gamma$ (Lemma \ref{embed})} (m-1-2)
(m-1-1) edgenode[auto] {$\bar{\rho}_{\scriptscriptstyle{\mathcal{G}}}$} (m-2-1)
(m-1-2) edgenode[auto] {$\rho_{\scriptscriptstyle{\mathcal{G}}}$ (Lemma \ref{maps})} (m-2-2)
(m-2-2) edgenode[auto] {$\theta$ (Prop.\ \ref{sumuniversecuntz})} (m-2-1);
\end{tikzpicture}
\end{center}
\end{cor}

\begin{proof}
Define $\rho_{\mathcal G}:D_p(\mathcal K) \to \mathcal O(\mathcal K)$ by Lemma \ref{maps} and $\gamma:D_p(\mathcal H) \to D_p^0(\mathcal K)$ from Lemma \ref{embed}.
Let $x=x_1 \otimes \cdots \otimes x_l$ and $y=y_1\otimes \cdots \otimes y_m$.
We compute
\begin{align*}
\rho_{\mathcal G}(\gamma(e_{x,y} \otimes 1_{\mathcal H^{\otimes k}}))
&= \rho_{\mathcal G}(\gamma(e_{x_1\otimes\cdots\otimes x_l,y_1\otimes\cdots\otimes y_m} \otimes 1_{\mathcal H^{\otimes k}})) \\
&= \rho_{\mathcal G}(e_{(x_1,0)\otimes\cdots\otimes (x_l,0),(y_1,0)\otimes\cdots\otimes (y_m,0)} \otimes 1_{\mathcal K^{\otimes k}}) \\
&= g_{k + l}S_{(x_1,0)\otimes\cdots\otimes (x_l,0)}S_{(y_1,0)\otimes\cdots\otimes(y_m,0)}^*g_{k + m},
\end{align*}
and we see that the result is in $C^*(\pi,\tau)$ where $(\pi,\tau)$ is as in Proposition \ref{sumuniversecuntz}.
Hence, the image of $\rho_{\mathcal G} \circ\gamma$ is contained in $C^*(\pi,\tau)$ and using $\theta:C^*(\pi,\tau) \to \mathcal T(\mathcal H)$ defined by Proposition \ref{sumuniversecuntz}, we may define 
\[ \bar\rho_{\mathcal G} = \theta \circ \rho_{\mathcal G} \circ \gamma:D_p{\mathcal H} \to \mathcal T(\mathcal H). \]
We further compute
\begin{align*}
\bar\rho_{\mathcal G}(\gamma(e_{x,y} \otimes 1_{\mathcal H^{\otimes k}}))
&= \theta(g_{k + l}S_{(x_1,0)\otimes\cdots\otimes (x_l,0)}S_{(y_1,0)\otimes\cdots\otimes(y_m,0)}^*g_{k + m}) \\
&= g_{k + l}T_{x_1\otimes \cdots\otimes x_l}T_{y_1\otimes\cdots\otimes y_m}^*g_{k + m},
\end{align*}
as required.
\end{proof}

Let $A$ be a $C^*$-algebra, let $\mathcal H$ be a countably generated correspondence over $A$.
Consider a tuple $\mathcal G=(g_0,\dots,g_{p-1})$ of positive contractions from $A_\infty$.
We may lift this to a sequence $(\mathcal G_i)$ of tuples of positive contractions from $A$.
Using this lift and Corollary \ref{spock_cor}, we define
\begin{equation}
\label{barrhoDef}
\bar\rho_{\mathcal G} = (\rho_{\mathcal G_i}):D_p(\mathcal H) \to \mathcal O(\mathcal H)_\infty.
\end{equation}
Note that this is independent of the choice of the lift $(\mathcal G_i)$.

\begin{cor}\label{lulu}
Let $A$ be a separable $C^*$-algebra, let $\mathcal H$ be a countably generated correspondence over $A$, and let $p \in \mathbb N$.
If $\mathcal{G} = (g_0,\ldots,g_{p-1})$ is a tuple of orthogonal positive contractions in $A_\infty$ satisfying $z \ldotp g_k = g_{k + \lvert z\rvert}\ldotp z$ for all $k$ and all elementary tensors $z\in \mathcal F_p(\mathcal H)$
then the c.p.\ map $\bar{\rho}_{\scriptscriptstyle{\mathcal{G}}}$ from \eqref{barrhoDef} is contractive and order zero.
\end{cor}

\begin{proof}
Analogously to \eqref{barrhoDef}, define $\rho_{\mathcal G} = (\rho_{\mathcal G_i}):D_p(\mathcal K) \to \mathcal O(\mathcal K)_\infty$ and $\sigma_{\mathcal G} = (\sigma_{\mathcal G_i}):D_p(\mathcal K) \to \mathbb B(\mathcal F(\mathcal K))_\infty$.
We have
\begin{align*}
\bar\rho_{\mathcal G} &= \theta_\infty \circ \rho_{\mathcal G} \circ \gamma \text{ and} \\
\rho_{\mathcal G} &= Q_\infty \circ \sigma_{\mathcal G}.
\end{align*}
By Lemma \ref{moremaps} (1) and the hypothesis,
\begin{equation*}
\lVert [R_{\mathcal G_i}^*R_{\mathcal G_i}, e_{w,z}\otimes 1_{\mathcal{K}^{\otimes k}}]\rVert \to 0
\end{equation*}
for all $\mathcal H$-elementary tensors $w,z\in \mathcal F(\mathcal{K})$. By the definition of $\gamma$ (from Lemma \ref{embed}), this implies that
\[ \lVert [R_{\mathcal G_i}^*R_{\mathcal G_i}, \gamma(x)] \rVert \to 0 \]
for all $x \in D_p(\mathcal H)$.
By the formula for $\sigma_{\mathcal G_i}$ (from Lemma \ref{maps}), it follows that $\sigma_{\mathcal G}$ is order zero.
Consequently, $\rho_{\mathcal G} \circ \gamma$, and therefore also $\bar\rho_{\mathcal G}$, is order zero.

To see that $\bar\rho_{\mathcal G}$ is contractive, we have by Lemma \ref{moremaps} (2) that
\[ \|R_{\mathcal G_i}^*R_{\mathcal G_i} - \sum_k g_{k,i}^2|_{\mathcal K^{\otimes k}}\| \to 0, \]
where $\mathcal G_i = (g_{0,i},\dots,g_{p-1,i})$.
Since the $g_i$ are orthogonal and contractive, $\sum_i g_i^2$ is contractive.
It follows from the formula for $\sigma_{\mathcal G_i}$ that $\sigma_{\mathcal G}$ is contractive.
Hence, so is $\bar\rho_{\mathcal G}$.
\end{proof}

\subsection{Proof of the main theorem}\label{Proof}
We have one technical lemma before the main result.
For $p \in \mathbb N$ and $k =0,\dots,p-1$, set
\begin{equation}
\label{d_kDef}
 d_p(k) = 1 - \frac{\lvert p-1-2k\rvert}{p-1} \in [0,1]
\end{equation}
\\
\begin{center}
\includegraphics[scale=.4]{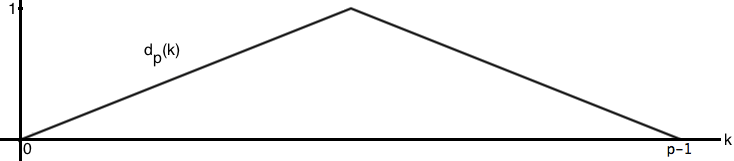}.
\end{center}

\begin{lem}\label{bumplemma}
Let $p$ be an odd integer.
If $f_0,\ldots f_{p-1}$ are positive contractions in a $C^*$-algebra $A$, then for any $0 \le N\le p-1$,
\begin{equation}\label{bump}
\sum_{k = N}^{p-1}d_p(k).\left(f_k + f_{\frac{p-1}{2} + k\ (\mathrm{mod}\ p)}\right) \ \approx_{\frac{4N^2}{p-1}} \sum_{k=0}^{p-1}f_k. 
\end{equation}
\end{lem}
\begin{proof}
It's easy to verify that 
\begin{equation*}
d_p(k) + d_p(\tfrac{p-1}{2} + k\ (\mathrm{mod}\ p)) = 1 \ \text{for}\ k =0,\dots,p-1
\end{equation*}
so that
\begin{equation*}
\sum_{k=0}^{p-1} d_p(k)\left(f_k + f_{\frac{p-1}{2} + k (\mathrm{mod}\ p)}\right) 
= \sum_{k=0}^{p-1}f_k.
\end{equation*}
Moreover, 
\begin{equation*}
\lVert\sum_{k=0}^{N-1}d_p(k)\left(f_k + f_{\frac{p-1}{2} + k}\right)\rVert \le 2\sum_{k=0}^{N-1}d_p(k) \le 2N\left(1 - \frac{p-1-2N}{p-1}\right) = \frac{4N^2}{p-1},
\end{equation*}
so the result follows. 
\end{proof}

We now have the ingredients to prove the main result. 

\begin{thm}\label{main_thm}
Suppose that $\mathcal{H}$ is a countably generated $C^*$-correspondence over a separable unital $C^*$-algebra $A$. Assume further that for every $p\in \mathbb{N}$, there is an approximate unit consisting of projections in $\mathbb{K}(\mathcal{F}_p(\mathcal{H}))$ that are quasicentral in $D_p(\mathcal{H})$. Then 
\begin{equation*}
(\dim_{\mathrm{nuc}}(\mathcal{T}(\mathcal{H})) + 1) \le 2(\dim_{\mathrm{nuc}}(A)+1)(\dim_{\mathrm{Rok}}(\mathcal{H})+1).
\end{equation*}
\end{thm}

\begin{rem}
Example \ref{QD-examples} gives two important examples of classes of correspondences which satisfy the technical hypothesis of this theorem, namely, that for every $p$, there is an approximate unit of projections in $\mathbb K(\mathcal F_p(\mathcal H))$ which is quasicentral in $D_p(\mathcal H)$.
\end{rem}

\begin{proof}
If either $\dim_{\mathrm{nuc}}(A)$ or $\dim_{\mathrm{Rok}}(\mathcal{H},A) = \infty$ there is nothing to show. Otherwise, let $\dim_{\mathrm{nuc}}(A) = n$ and $\dim_{\mathrm{Rok}}(\mathcal{H}) = d$. We will use Lemma \ref{nucdimtool} with $m=d+1$.
Therefore let $F$ be a finite subset of $\mathcal{T}(\mathcal{H})$ and fix $\epsilon > 0$.

Noting \eqref{Tspan}, we may assume without loss of generality that $F$ consists of elements of the form $T_xT_y^*$ where $x,y$ are elementary tensors.
Pick $N$ such that every element of $F$ is $T_xT_y^*$ for some elementary tensors $x,y \in \mathcal F_{N+1}(\mathcal H)$.
Let $p \in \mathbb{N}$ be large enough so that 
\begin{equation}\label{pestimate}
(d+1)\left(2\sqrt{\frac{2N}{p-1}} + \frac{4N^2}{p-1}\right) < \epsilon
\end{equation}
and let $P\in \mathbb{B}(\mathcal{F}(\mathcal{H}))$ be the projection onto the $p^{\text{th}}$-cutoff Fock space $\mathcal{F}_p(\mathcal{H})$. 

For $k=0,\dots,p-1$, use $d_p(k)$ as defined in \eqref{d_kDef}; for convenience, we set $d_p(k)=0$ for $k \geq p$.
Using these, set $\Delta = \text{diag}(d_p(0),\ldots,d_p(p-1))$ and let $\phi: \mathcal{T}(\mathcal{H}) \rightarrow  \mathbb B(\mathcal F_p(\mathcal H))$ be compression by $\sqrt{\Delta} P$; that is, for elementary tensors $\zeta,x,y$,
\[
\phi(T_xT_y^*)(\zeta) =
\begin{cases}
d_p(|\zeta|)d_p(|\zeta|+|x|-|y|) x \otimes \langle y,\zeta_1\rangle \ldotp \zeta_2, \quad &\zeta=\zeta_1\otimes \zeta_2, |\zeta_1|=|y|; \\ 0, \quad &|\zeta|<|y| \end{cases}
\]
(here it is convenient that $d_p(k)=0$ for $k \geq p$).
It is not hard to see that $\phi(\mathcal T(\mathcal H))$ is contained in $D_p(\mathcal H)$.

Find Rokhlin contractions $\{f_k^l\}_{l = 0,\ldots,d;\ k \in \mathbb Z/p}$ in $A_\infty$ satisfying 
\begin{enumerate}
\item $f_k^lf_{k'}^l=0$ for all $l$ and $k\ne k'$,
\item $\sum_{k,l}f_k^l =1$,
\item $z\ldotp f_k^l = f_{k + \lvert z\rvert}^l\ldotp z$ for all $k,l$, and all elementary tensors $z \in \mathcal F(\mathcal H)$. 
\end{enumerate}
Note that (3) implies that $z\ldotp (f_k^l)^m = (f_{k + \lvert z\rvert}^l)^m\ldotp z$ for all $m \in \mathbb N$, which in turn implies that $z\ldotp g(f_k^l) = g(f_{k + \lvert z\rvert}^l)\ldotp z$ for any continuous function $g \in C([0,1])$.
In particular,
\begin{equation}
\label{Rokhlinsqrt}
z\ldotp (f_k^l)^{1/2} = (f_{k + \lvert z\rvert}^l)^{1/2}\ldotp z
\end{equation}
For $l=0,\dots,d$, set 
\begin{equation*}
\mathcal{G}^l = ((f_0^l)^{1/2},\ldots,(f_{p-1}^l)^{1/2}) \ \ \ \text{and} \ \ \ \hat{\mathcal{G}}^l = ((f_{\frac{p-1}{2} + 0}^l)^{1/2},\ldots,(f_{\frac{p-1}{2} + p-1}^l)^{1/2}),
\end{equation*}
two tuples of orthogonal positive elements in $A_\infty$.
Define $\bar{\rho}_{\scriptscriptstyle{\mathcal{G}^l}}$ and $\bar{\rho}_{\scriptscriptstyle{\hat{\mathcal{G}}^l}}$ as in Corollary \ref{spock_cor}; using \eqref{Rokhlinsqrt}, Corollary \ref{lulu} tells us that these maps are c.p.c.\ and order zero.

To simplify notation, write $\rho^l$ and $\hat{\rho}^l$ for $\bar{\rho}_{\scriptscriptstyle{\mathcal{G}^l}}$ and $\bar{\rho}_{\scriptscriptstyle{\hat{\mathcal{G}}^l}}$, respectively.
Thus we have the following diagram.

\begin{center}
\begin{tikzpicture}
\matrix (m) [matrix of math nodes, row sep=2em,column sep=8em, text height=1.5ex, text depth=0.25ex]
{
\mathcal T(\mathcal H) & & \mathcal T(\mathcal H)_\infty \\
& D_p(\mathcal H) & \\
};
\path[->,font=\small]
(m-1-1) edgenode[auto] {} (m-1-3)
(m-1-1) edgenode[auto,swap] {$\phi$} (m-2-2)
(m-2-2) edgenode[auto,swap] {$\sum_l (\rho_l + \hat\rho_l)$} (m-1-3);
\end{tikzpicture}
\end{center}

By Lemma \ref{quasi_hypothesis}, $D_p(\mathcal H)$ has nuclear dimension at most $n$.
We shall show that this diagram commutes, up to $\epsilon$ on $F$.
As $\phi$ is c.p.c.\ and each $\sum_l (\rho_l+\hat\rho_l)$ is overtly a sum of $2(d+1)$ c.p.c.\ order zero maps, this will finish verifying the hypotheses of Lemma \ref{nucdimtool}, and thus finish the proof.

Consider an element of $F$, necessarily of the form $T_xT_y^*$, $x,y \in \mathcal F_{N+1}(\mathcal H)$.
For the moment, assume that $\lvert x\rvert \ge \lvert y\rvert$.  Applying $\phi$ to $T_xT_y^*$,  using \eqref{Strict-Expansion}, and noting that  $\sqrt{d_p(k +\lvert x\rvert)}$ is within $\sqrt{\frac{2N}{p-1}}$ of $\sqrt{d_p(k + \lvert y\rvert)}$ for $k \in \mathbb Z/p$, we have
\begin{align}
\phi(T_xT_y^*) &= \sum_{k = 0}^{p-1- \lvert x\rvert}\sqrt{d_p(k + \lvert x\rvert)d_p(k + \lvert y\rvert)}e_{x,y}\otimes 1_{\mathcal{H}^{\otimes k}}\notag\\
\label{equationcutdown}
& \approx_{\sqrt{\frac{2N}{p-1}}} \sum_{k=0}^{p-1-\lvert x\rvert}d_p(k+\lvert x\rvert)e_{x,y}\otimes 1_{\mathcal{H}^{\otimes k}},
\end{align}
where the error estimate in the last line is a maximum (rather than a sum) because of orthogonality.

Next, applying $\rho^l$ to $\phi(T_xT_y^*)$, we get 
\begin{align}
\notag
\rho^{l}\circ\phi(T_xT_y^*) 
&\stackrel{\eqref{equationcutdown}}\approx_{\sqrt{\frac{2N}{p-1}}} \sum_{k=0}^{p-1-\lvert x\rvert} d_p(k + \lvert x\rvert) \rho^l(e_{x,y} \otimes 1_{\mathcal H^{\otimes k}}) \\
\notag
&\stackrel{\eqref{barrhoeq}}= \sum_{k=0}^{p-1-\lvert x\rvert}d_p(k + \lvert x\rvert) (f_{k + \lvert x\rvert}^l)^{1/2}T_xT_y^*(f_{k + \lvert y\rvert}^l)^{1/2} \\
\notag
&= \sum_{k=\lvert x\rvert}^{p-1}d_p(k) (f_{k}^l)^{1/2}T_xT_y^*(f_{k - \lvert x\rvert+ \lvert y\rvert}^l)^{1/2} \\
\notag
&= \sum_{k=\lvert x\rvert}^{p-1}d_p(k) (f_{k}^l)^{1/2}T_x(f_{k - \lvert x\rvert}^l)^{1/2}T_y^* \\
\label{rhol-eq1}
&=  (\sum_{k=\lvert x\rvert}^{p-1}d_p(k)f_k^l) T_xT_y^*,
\end{align}
and likewise,
\begin{equation}
\label{rhol-eq2}
\hat{\rho}^l\circ\phi(T_xT_y^*) \approx_{\sqrt{\frac{2N}{p-1}}} (\sum_{k=\lvert x\rvert}^{p-1}d_p(k)f_{\frac{p-1}{2} + k}^l) T_xT_y^*.
\end{equation}
Summing these terms, we obtain
\begin{eqnarray*}
\sum_{l=0}^d (\rho^l + \hat{\rho}^l)\circ \phi(T_xT_y^*) \hspace*{-3em}\,
&\stackrel{\eqref{rhol-eq1},\eqref{rhol-eq2}}{_{\phantom{2(d+1)\sqrt{\frac{2N}{p-1}}}}\approx_{2(d+1)\sqrt{\frac{2N}{p-1}}}}& \sum_{l=0}^d \sum_{k=|x|}^{p-1} d_p(k) (f_k^l + f_{\frac{p-1}2+k}^l)T_xT_y^* \\
&\stackrel{\text{Lemma \ref{bumplemma}}}{_{\phantom{{(d+1)\frac{4N^2}{p-1}}}}\approx_{(d+1)\frac{4N^2}{p-1}}}& \sum_{l=0}^d \sum_{k=0}^{p-1} f_k^l T_xT_y^* \\
&=& T_xT_y^*.
\end{eqnarray*}
By \eqref{pestimate}, this yields
\[ \sum_{l=0}^d (\rho^l + \hat{\rho}^l)\circ \phi(T_xT_y^*) \approx_\epsilon T_xT_y^*, \]
as required.
A nearly identical argument shows the same estimate in the case $\lvert x\rvert < \lvert y\rvert$. 
\end{proof}

\begin{cor}\label{Main-Theorem_Cor}
Under the same hypotheses as the previous theorem,
\begin{equation*}
(\dim_{\mathrm{nuc}}(\mathcal{O}(\mathcal{H})) + 1) \le 2(\dim_{\mathrm{nuc}}(A)+1)(\dim_{\mathrm{Rok}}(\mathcal{H})+1).
\end{equation*}
\end{cor}

\begin{proof}
This follows from the fact that $\mathcal{O}(\mathcal{H})$ is a quotient of $\mathcal{T}(\mathcal{H})$ by the ideal $\mathbb{K}(\mathcal{F}(\mathcal{H})I_{\mathcal{H}})$, and by \cite[Proposition 2.3]{Winter-Zacharias2}.
\end{proof}
\section{Nuclear dimension of certain free products}\label{Amalgamated}
We now use the results given in the previous sections to deduce finite nuclear dimension of a certain class of reduced amalgamated free products. For a more comprehensive treatment of free products, see the monograph \cite{Voiculescu-Dykema-Nica}. We start with a well-known example. 

\begin{example}\label{free-group}
There is a *-isomorphism 
\begin{equation*}
C(\mathbb{T})\text{\large{\textasteriskcentered}} C(\mathbb{T}) \cong C^*_r(\mathbb{F}_2).
\end{equation*}
Here the free product is being taken with respect to the state $f\mapsto \int_{\mathbb{T}}f(z)dz$ on $C(\mathbb{T})$. 
\end{example}

The free group $\mathbb{F}_2$ on two generators is not amenable, so the reduced group $C^*$-algebra is not nuclear and in particular has infinite nuclear dimension.  What this shows is that unlike other canonical $C^*$-constructions, finite nuclear dimension is not in general preserved under the reduced amalgamated free product construction, even in the abelian case. However, the next result by Speicher from \cite{Speicher} shows that there are exceptions.  

\begin{prop}\label{Speicher}
Let $\mathcal{H}_i$ be correspondences over $A$. Then 
\begin{equation*}
(\mathcal{T}(\bigoplus_I\mathcal{H}_i), E_{\bigoplus \mathcal{H}_i})  \cong \text{\Large{\textasteriskcentered}}_A(\mathcal{T}(\mathcal{H}_i), E_{\mathcal{H}_i}). 
\end{equation*}
\end{prop}

\begin{prop}\label{2-Main-Theorem}
Let $\mathcal{H}$ be a correspondence over $A$ and let $I$ be an arbitrary (countable) set. If $\dim_{\mathrm{Rok}}(\mathcal{H}) < \infty$, then $\dim_{\mathrm{Rok}}(\bigoplus_I\mathcal{H}) < \infty$.  
\end{prop}

\begin{proof}
The left action of $A$ on $\bigoplus_I\mathcal{H}$ is given by $a\ldotp (x_i)_{i\in I} = (a\ldotp x_i)_{i\in I}$. Let $\epsilon > 0$, $p \in \mathbb{N}$, $F$ a finite subset of $A$ and $\mathcal{V}$ a finite subset of $\bigoplus_I\mathcal{H}$.
Since the finitely supported elements of $\bigoplus_I \mathcal H$ are dense, we may assume that $\mathcal V$ consists only of finitely supported elements; let $N\in \mathbb N$ be such that each is supported on at most $N$ elements.
Set $\mathcal V'$ equal to the set of all elements in $\mathcal H$ which appear in some component of some element of $\mathcal V$; this is a finite set.
Find Rokhlin contractions $(f_k^l)_{k\in \mathbb Z/p}^{l=0,\ldots,d} \in A$ for the correspondence $\mathcal H$, with respect to $(\epsilon/N,p,F,\mathcal{V}')$.
For $x=(x_i)_{i\in I} \in \mathcal V$, 
\[ f_k^l\ldotp (x_i) = (f_k^l\ldotp x_i) \approx_{N\epsilon/N} (x_i\ldotp f_{k+1}^l) = (x_i) \ldotp f_{k+1}^l. \]
Thus, the $f_k^l$ are Rokhlin contractions for $\bigoplus_I \mathcal H$, with respect to $(\epsilon,p,F,\mathcal V)$.
\end{proof}

\begin{question}
For countably generated Hilbert modules $\mathcal H_1$ and $\mathcal H_2$, is it the case that
\[ \dim_{\mathrm{Rok}}(\mathcal H_1 \oplus \mathcal H_2) = \max\{ \dim_{\mathrm{Rok}}(\mathcal H_1), \dim_{\mathrm{Rok}}(\mathcal H_2)\}? \]
\end{question}

In light of Theorem \ref{Main-Theorem} and Proposition \ref{2-Main-Theorem}, we obtain the following statement. 

\begin{thm}\label{Second-Main-Theorem}
Let $\mathcal{H}$ be a finitely generated projective correspondence over $A$. If $\dim_{\mathrm{nuc}}(A)$ and $\dim_{\mathrm{Rok}}(\mathcal{H})$ are both finite, then the amalgamated free product of any finite number of copies of $\mathcal{T}(\mathcal{H})$ (with respect to the canonical expectation $E_{\mathcal{H}}$) has finite nuclear dimension. 
\end{thm}

In comparison with Example \ref{free-group}, the following example demonstrates an interesting consequence of Theorem \ref{Second-Main-Theorem}. 

\begin{example}
If $\varphi$ is a minimal homeomorphism of $\mathbb{T}$, then $\dim_{\mathrm{Rok}}(\varphi^*)$ is finite. By Theorem \ref{Second-Main-Theorem}, the nuclear dimension of 
\begin{equation*}
\mathcal{T}(C(\mathbb{T})^{\varphi^*})\text{\large{\textasteriskcentered}}\mathcal{T}(C(\mathbb{T})^{\varphi^*}) \cong \mathcal{T}(C(\mathbb{T})\oplus C(\mathbb{T}))
\end{equation*}
is also finite. Here the free product is being taken with respect to the usual conditional expectation $E_{C(\mathbb{T})\oplus C(\mathbb{T})}$. 
\end{example}
\section{Classifiability of certain Cuntz--Pimsner algebras}\label{classifiability}

In this section we show that in the presence of finite Rokhlin dimension, the algebras $\mathcal{O}(\mathcal{H})$ are often \emph{classifiable}, by which we mean unital and simple with finite nuclear dimension and satisfying the UCT (cf.\ \cite{Tikuisis-White-Winter}).
``Satisfying the UCT'' is a $KK$-theoretic property for separable $C^*$-algebras, equivalent to being $KK$-equivalent to an abelian $C^*$-algebra, introduced by Rosenberg and Schochet in \cite{Rosenberg-Schochet}.
The class of algebras which satisfy the UCT (sometimes called the \emph{bootstrap class}) is closed under a number of natural operations (see \cite[\S22.3 and \S23]{Blackadar:KBook}), although it is unknown whether every separable nuclear $C^*$-algebra is in this class.

\begin{lem}\label{classif_lem_1}
For simple $C^*$-algebras, if $A$ satisfies the UCT, so does $\mathcal{O}(\mathcal{H})$. 
\end{lem}

\begin{proof}
By \cite[Theorem 4.4]{Pimsner}, the Toeplitz Pimsner algebra $\mathcal{T}(\mathcal{H})$ is $KK$-equivalent to $A$. Consequently, $\mathcal{T}(\mathcal{H})$ satisfies the UCT whenever $A$ does. Since $\mathcal{O}(\mathcal{H})$ is a quotient of $\mathcal{T}(\mathcal{H})$ by the ideal $\mathbb{K}(\mathcal{\mathcal{F}(\mathcal{H})I_{\mathcal{H}}})$, it suffices to prove (by the two-out-of-three principle) that $\mathbb{K}(\mathcal{\mathcal{F}(\mathcal{H})I_{\mathcal{H}}})$ satisfies the UCT. To this end observe that $\mathcal{\mathcal{F}(\mathcal{H})I_{\mathcal{H}}}$ is countably generated, so Kasparov's stabilization result implies 
\begin{equation*}
\mathbb{K}(\mathcal{\mathcal{F}(\mathcal{H})I_{\mathcal{H}}}\oplus\ell^2(A))\cong \mathbb{K}(\ell^2(A)) \cong A\otimes\mathbb{K}\sim_{\mathrm{Morita}} A,
\end{equation*}
where $\sim_{\mathrm{Morita}}$ denotes Morita equivalence of $C^*$-algebras. Since $A$ is simple, so is $A\otimes \mathbb{K}\cong \mathbb{K}(\mathcal{\mathcal{F}(\mathcal{H})I_{\mathcal{H}}}\oplus\ell^2(A))$. This means every hereditary $C^*$ subalgebra $B\subseteq \mathbb{K}(\mathcal{\mathcal{F}(\mathcal{H})I_{\mathcal{H}}}\oplus\ell^2(A))$ is full, and hence satisfies
\begin{equation*}
B\sim_{\mathrm{Morita}} \mathbb{K}(\mathcal{\mathcal{F}(\mathcal{H})I_{\mathcal{H}}}\oplus\ell^2(A))\sim_{\mathrm{Morita}} A. 
\end{equation*}
The result follows because $\mathbb{K}(\mathcal{\mathcal{F}(\mathcal{H})I_{\mathcal{H}}})$ is hereditary in $\mathbb{K}(\mathcal{\mathcal{F}(\mathcal{H})I_{\mathcal{H}}}\oplus\ell^2(A))$, and because the UCT is preserved under Morita equivalence. 
\end{proof}

Next, following Schweizer, we deal with simplicity of $\mathcal{O}(\mathcal{H})$. 

\begin{definition}
$\mathcal{H}$ is called \emph{minimal} if there are no nontrivial ideals $J \subset A$ such that $\{\langle x,j\ldotp y\rangle \ : \ x,y\in \mathcal{H},\ j\in J\}\subseteq J$.  $\mathcal{H}$ is \emph{nonperiodic} if $\mathcal{H}^{\otimes k}\approx A \Rightarrow k = 0$, where we regard $A$ as the identity correspondence over itself. 
\end{definition}

Generalizing the well-known fact that $C(X)\rtimes_{\sigma}\mathbb{Z}$ is simple if and only if the homeomorphism $\sigma$ is minimal, Schweizer proved the following beautiful result. 

\begin{thm}[\cite{Schweizer}]\label{thm:schweizer}
$\mathcal{O}(\mathcal{H})$ is simple if and only if $\mathcal{H}$ is minimal and nonperiodic. 
\end{thm}

Note that minimality is automatic whenever the scalar algebra $A$ is simple, so in this case, we are left to worry about periodicity. 

\begin{lem}\label{classif_lem_2}
If $\mathcal{H}$ has finite Rokhlin dimension, it is nonperiodic. 
\end{lem}

\begin{proof}
Assume $\dim_{\text{Rok}}(\mathcal{H}) \le d$ and that there is some $k > 0$ such that $\mathcal{H}^{\otimes k}\approx A$. This means there is an adjointable unitary bimodule map $U:\mathcal{H}^{\otimes k} \rightarrow A$. Let $v = U^{-1}(1_A)$ and fix $\epsilon > 0$. It is straightforward to check that $\dim_{\text{Rok}}(\mathcal{H}) \le d$ implies $\dim_{\text{Rok}}(\mathcal{H}^{\otimes k}) \le d$, so we can find a set of positive contractions $\{f_1^l,f_2^l\}_{l = 0,\ldots,d}\subset A$ satisfying the following three estimates:
\begin{enumerate}
\item $\lVert f_1^lf_2^l\rVert < \epsilon$ for every $l$,
\item $\lVert\sum_{l=0}^d(f_1^l + f_2^l) - 1_A\rVert < \epsilon$, and 
\item $\lVert v\ldotp f_1^l - f_2^l\ldotp v\rVert < \epsilon$ for every $l$. 
\end{enumerate}
Applying $U$ to the third estimate gives $\lVert f_1^l - f_2^l\rVert < \epsilon$ for every $l$. Combining this with the first estimate gives $\lVert f_i^l\rVert < \sqrt{2\epsilon}$ for $i = 1,2$. But this implies $\lVert \sum_{l=0}^d f_1^l + f_2^l\rVert < 2(d+1)\sqrt{2\epsilon}$, which contradicts the second estimate when $\epsilon$ is sufficiently small. Thus, $\mathcal{H}$ must be nonperiodic.
\end{proof}

Combining Theorem \ref{thm:schweizer} and the previous lemma we see that for $C^*$-correspondences with finite Rokhlin dimension, $\mathcal{O}(\mathcal{H})$ is simple and unital whenever $A$ is. For simple $A$, by Lemma \ref{classif_lem_1}, $\mathcal{O}(\mathcal{H})$ satisfies the UCT whenever $A$ does. Thus applying Theorem \ref{main_thm}, we obtain the following corollary.

\begin{cor}
\label{application2}
If $A$ is classifiable, $\dim_{\text{Rok}}(\mathcal{H}) <\infty$ and for every $p\in \mathbb{N}$ there is an approximate unit consisting of projections in $\mathbb{K}(\mathcal{F}_p(\mathcal{H}))$ that are quasicentral in $D_p(\mathcal{H})$ (e.g., if $\mathcal{H}$ is finitely generated projective), then $\mathcal{O}(\mathcal{H})$ is classifiable. 
\end{cor}

\begin{example}
Suppose $A$ is a unital Kirchberg algebra and $\alpha:A\rightarrow A$ is an automorphism with finite Rokhlin dimension (such automorphisms are generic by \cite[Theorem 3.4]{Hirshberg-Winter-Zacharias}). If $\mathcal{H}$ is a countably generated free $A$-module with basis $(\xi_i)_{i=1}^{\infty}$, we can define a correspondence with the left action $a\ldotp \xi_i = \xi_i\ldotp \alpha(a)$. Note that $A\cap \mathbb{K}(\mathcal{H}) = \{0\}$, so Remark \ref{Cuntz=Pimsner} and the $KK$-equivalence of $A$ and $\mathcal{T}(\mathcal{H})$ imply $A$ is $KK$-equivalent to $\mathcal{O}(\mathcal{H})$. Moreover, $\dim_{\text{Rok}}(\mathcal{H}) < \infty$, so $\mathcal{O}(\mathcal{H})$ is a simple, separable $C^*$-algebra with finite nuclear dimension. Since the creation operator associated to $\xi_1$ is a proper isometry, $\mathcal{O}(\mathcal{H})$ must be a Kirchberg algebra. Invoking Kirchberg--Phillips classification, we conclude that $\mathcal{O}(\mathcal{H})\cong A$. 
\end{example}

We end this paper by pointing out that our work does not recapture the interesting examples of classifiable Cuntz--Pimsner algebras constructed by Kumjian in \cite{Kumjian}. Indeed, it can be shown that they don't have finite Rokhlin dimension. 

\bibliographystyle{amsplain}
\providecommand{\bysame}{\leavevmode\hbox to3em{\hrulefill}\thinspace}

\end{document}